\documentclass[11pt]{article}

\usepackage[nottoc,notlot,notlof]{tocbibind}

\usepackage{amsmath,amssymb,amsfonts,amsthm}
\usepackage{latexsym}
\usepackage{tikz}
\usetikzlibrary{automata,positioning}
\usetikzlibrary{chains,fit,shapes}
\usetikzlibrary{calc}
\usetikzlibrary{arrows}
\usepackage{tkz-graph}
\usetikzlibrary{positioning,calc}
\usetikzlibrary{graphs}
\usetikzlibrary{graphs.standard}
\usetikzlibrary{arrows,decorations.markings}

\newcount\nodecount
\tikzgraphsset{
  declare={subgraph N}%
  {
    [/utils/exec={\global\nodecount=0}]
    \foreach \nodetext in \tikzgraphV
    {  [/utils/exec={\global\advance\nodecount by1}, 
      parse/.expand once={\the\nodecount/\nodetext}] }
  },
  declare={subgraph C}%
  {
    [cycle, /utils/exec={\global\nodecount=0}]
    \foreach \nodetext in \tikzgraphV
    {  [/utils/exec={\global\advance\nodecount by1}, 
      parse/.expand once={\the\nodecount/\nodetext}] }
  }
}

\newtheorem{definition}{Definition}
\newtheorem{theorem}{Theorem}

\newtheorem{remark}[theorem]{Remark}
\newtheorem{proposition}[theorem]{Proposition}
\newtheorem{corollary}[theorem]{Corollary}

\newcommand\p{\circle*{0.3}}

\title{A Comprehensive Introduction to the Theory of Word-Representable Graphs}

\author{Sergey Kitaev}

\begin{document}  

\maketitle

\abstract{Letters $x$ and $y$ alternate in a word $w$ if after deleting in $w$ all letters but the copies of $x$ and $y$ we either obtain a word $xyxy\cdots$ (of even or odd length) or a word $yxyx\cdots$ (of even or odd length). A graph $G=(V,E)$ is word-representable if and only if there exists a word $w$
over the alphabet $V$ such that letters $x$ and $y$ alternate in $w$ if and only if $xy\in E$.

Word-representable graphs generalize several important classes of graphs such as circle graphs, $3$-colorable graphs and comparability graphs. This paper offers a comprehensive introduction to the theory of word-representable graphs including the most recent developments in the area.}  

\tableofcontents

\section{Introduction}\label{sec1}

The theory of word-representable graphs is a young but very promising research area.  It was introduced by the author in 2004 based on the joint research with Steven Seif \cite{KS08} on the celebrated {\em Perkins semigroup}, which has played a central role in semigroup theory since 1960, particularly as a source of examples and counterexamples. However, the first systematic study of word-representable graphs was not undertaken until the appearance in 2008 of the paper \cite{KP08} by the author and Artem Pyatkin, which started the development of the theory. One of the most significant contributors to the area is Magn\'us M. Halld\'orsson.

Up to date, nearly 20 papers have been written on the subject, and the core of the book \cite{KL15} by the author and Vadim Lozin is devoted to the theory of word-representable graphs. It should also be mentioned that the software produced by Marc Glen \cite{G} is often of great help in dealing with word-representation of graphs.

We refer the Reader to \cite{KL15}, where relevance of word-representable graphs to various fields is explained, thus providing a motivation to study the graphs. These fields are algebra, graph theory, computer science, combinatorics on words, and scheduling. In particular, word-representable graphs are important from graph-theoretical point of view, since they generalize several fundamental classes of graphs (e.g.\ {\em circle graphs}, {\em $3$-colorable graphs} and {\em comparability graphs}).

A graph $G=(V,E)$ is {\em word-representable} if and only if there exists a word $w$
over the alphabet $V$ such that letters $x$ and $y$, $x\neq y$, alternate in $w$ if and only if $xy\in E$ (see Section~\ref{sec2} for the definition of alternating letters). Natural questions to ask about word-representable graphs are:

\begin{itemize}
\item Are all graphs word-representable?
\item If not, how do we characterize word-representable graphs?
\item How many word-representable graphs are there?
\item What is graph's representation number for a given graph? Essentially, what is the minimal length of a word-representant?
\item How hard is it to decide whether a graph is word-representable or not? (complexity) 
\item Which graph operations preserve (non-)word-representability?  
\item Which graphs are word-representable in your favourite class of graphs?  
\end{itemize}

This paper offers a comprehensive introduction to the theory of word-representable graphs. Even though the paper is based on the book~\cite{KL15} following some of its structure, our exposition goes far beyond book's content and it reflects the most recent developments in the area. Having said that, there is a relevant topic on a generalization of the theory of word-representable graphs \cite{JKPR15} that is discussed in \cite[Chapter 6]{KL15}, but we do not discuss it at all.

In this paper we do not include the majority of proofs due to space limitations (while still giving some proofs, or ideas of proofs whenever possible). Also, all graphs we deal with are simple (no loops or multiple edges are allowed), and unless otherwise specified, our graphs are unoriented.

\section{Word-Representable Graphs. The Basics}\label{sec2}

Suppose that $w$ is a word over some alphabet and $x$ and $y$ are two distinct letters in $w$. We say that $x$ and $y$ {\em alternate} in $w$ if after deleting in $w$ {\em all} letters {\em but} the copies of $x$ and $y$ we either obtain a word $xyxy\cdots$ (of even or odd length) or a word $yxyx\cdots$ (of even or odd length). For example, in the word 23125413241362, the letters 2 and 3 alternate. So do the letters 5 and 6, while the letters 1 and 3 do {\em not} alternate.  

\begin{definition}\label{wrg-def} A graph $G=(V,E)$ is {\em word-representable} if and only if there exists a word $w$
over the alphabet $V$ such that letters $x$ and $y$, $x\neq y$, alternate in $w$ if and only if $xy\in E$. (By definition, $w$ {\em must} contain {\em each} letter in $V$.) We say that $w$ {\em represents} $G$, and that $w$ is a {\em word-representant}. \end{definition}

Definition~\ref{wrg-def} works for both vertex-labeled and unlabeled graphs because any labeling of a graph $G$ is equivalent to any other labeling of $G$ with respect to word-representability (indeed, the letters of a word $w$ representing $G$ can always be renamed). For example, the graph to the left in Figure~\ref{wrg-ex} is word-representable because its labeled version to the right in Figure~\ref{wrg-ex} can be represented by 1213423. For another example, each {\em complete graph} $K_n$ can be represented by any permutation $\pi$ of $\{1,2,\ldots,n\}$, or by $\pi$ concatenated any number of times.   Also, the {\em empty graph} $E_n$ (also known as {\em edgeless graph}, or {\em null graph}) on vertices $\{1,2,\ldots,n\}$ can be represented by $12\cdots (n-1)nn(n-1)\cdots 21$, or by any other permutation concatenated with the same permutation written in the reverse order.

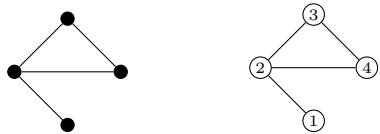
\begin{figure}[h]
\begin{center}
\begin{tabular}{ccc}
\begin{tikzpicture}[node distance=1cm,auto,main node/.style={fill,circle,draw,inner sep=0pt,minimum size=5pt}]

\node[main node] (1) {};
\node[main node] (2) [below left of=1] {};
\node[main node] (3) [below right of=1] {};
\node[main node] (4) [below right of=2] {};

\path
(1) edge (2)
(1) edge (3);

\path
(2) edge (3)
(2) edge (4);
\end{tikzpicture}

& 

\ \ \ \ \ \ \

&

\begin{tikzpicture}[node distance=1cm,auto,main node/.style={circle,draw,inner sep=1pt,minimum size=2pt}]

\node[main node] (1) {{\tiny 3}};
\node[main node] (2) [below left of=1] {{\tiny 2}};
\node[main node] (3) [below right of=1] {{\tiny 4}};
\node[main node] (4) [below right of=2] {{\tiny 1}};

\path
(1) edge (2)
(1) edge (3);

\path
(2) edge (3)
(2) edge (4);

\end{tikzpicture}

\end{tabular}
\end{center}
\vspace{-5mm}
\caption{An example of a word-representable graph}\label{wrg-ex}
\end{figure}

\begin{remark}\label{hereditary} The class of word-representable graphs is {\em hereditary}. That is, removing a vertex $v$ in a word-representable graph $G$ results in a word-representable graph $G'$. Indeed, if $w$ represents $G$ then $w$ with $v$ removed represents $G'$. This observation is crucial, e.g.\ in finding asymptotics for the number of word-representable graphs~\cite{CKL17}, which is the only known enumerative result on word-representable graphs to be stated next. \end{remark}
 
\begin{theorem}[\cite{CKL17}] The number of non-isomorphic word-representable graphs on $n$ vertices is given by $2^{\frac{n^2}{3}+o(n^2)}$.\end{theorem}

\subsection{$k$-Representability and Graph's Representation Number}

A word $w$ is {\em $k$-uniform} if each letter in $w$ occurs $k$ times. For example, 243321442311 is a 3-uniform word, while 23154 is a 1-uniform word (a permutation). 

\begin{definition}\label{k-repr-def} A graph $G$ is {\em $k$-word-representable}, or {\em $k$-representable} for brevity, if there exists a $k$-uniform word $w$ representing it. We say that $w$ {\em $k$-represents}  $G$. \end{definition}

The following result establishes equivalence of Definitions~\ref{wrg-def} and~\ref{k-repr-def}.

\begin{theorem}[\cite{KP08}]\label{equiv-thm} A graph is word-representable if and only if it is $k$-representable for some $k$. \end{theorem}

\begin{proof} Clearly, $k$-representability implies word-representability. For the other direction, we demonstrate on an example how to extend a word-representant to a uniform word representing the same graph. We refer to \cite{KP08} for a precise description of the extending algorithm, and an argument justifying it. 

Consider the word $w=3412132154$ representing a graph $G$ on five vertices. Ignore the letter 1 occurring the maximum number of times (in general, there could be several such letters all of which need to be ignored) and consider the {\em initial permutation} $p(w)$ of $w$ formed by the remaining letters, that is, $p(w)$ records the order of occurrences of the leftmost copies of the letters. For our example, $p(w)=3425$. Then the word $p(w)w=34253412132154$ also represents $G$, but it contains more occurrences of the letters occurring not maximum number of time in $w$. This process can be repeated a number of times until each letter occurs the same number of times. In our example, we need to apply the process one more time by appending 5, the initial permutation of $p(w)w$, to the left of $p(w)w$ to obtain a uniform representation of $G$:  534253412132154.\end{proof}

Following the same arguments as in Theorem~\ref{equiv-thm}, one can prove the following result showing that there are {\em infinitely many} representations for any word-representable graph.

\begin{theorem}[\cite{KP08}]\label{k-implies-k-plus-1} If a graph is $k$-representable
  then it is also $(k+1)$-representable. \end{theorem}

By Theorem~\ref{equiv-thm}, the following notion is well-defined. 

\begin{definition} {\em Graph's representation number} is the {\em least} $k$ such that the graph is $k$-representable. For non-word-representable graphs (whose existence will be discussed below), we let $k=\infty$. Also, we let $\mathcal{R}(G)$ denote $G$'s representation number and $\mathcal{R}_k=\{ G : \mathcal{R}(G)=k\}$.\end{definition}

Clearly, $\mathcal{R}_1=\{ G : G \mbox{ is a complete graph}\}$. Next, we discuss 
$\mathcal{R}_2$. 

\subsection{Graphs with Representation Number 2}

We begin with discussing five particular classes of graphs having representation number 2, namely, {\em empty graphs}, {\em trees}, {\em forests}, {\em cycle graphs} and {\em ladder graphs}. Then we state a result saying that graphs with representation number 2 are exactly the class of {\em circle graphs}. 

\subsubsection{Empty Graphs}\label{empty-sec} No empty graph $E_n$ for $n\geq 2$ can be represented by a single copy of each letter, so  $\mathcal{R}(E_n)\geq 2$. On the other hand, as discussed above, $E_n$ can be represented by concatenation of two permutations, and thus $\mathcal{R}(E_n)= 2$.

\subsubsection{Trees and Forests} 
A simple inductive argument shows that {\em any tree} $T$ can be represented using two copies of each letter, and thus, if the number of vertices in $T$ is at least $3$, $\mathcal{R}(T)= 2$. Indeed, as the base case we have the edge labeled by 1 and 2 that can be 2-represented by 1212. Now, suppose that any tree on at most $n-1$ vertices can be 2-represented for $n\geq 3$, and consider a tree $T$ with $n$ vertices and with a leaf $x$ connected to a vertex $y$. Removing the leaf $x$, we obtain a tree $T'$ that can be 2-represented by a word $w_1yw_2yw_3$ where $w_1$, $w_2$ and $w_3$ are possibly empty words not containing $y$. It is now easy to see that the word $w_1yw_2xyxw_3$ 2-represents $T$ (obtained from $T'$ by inserting back the leaf $x$). Note that the word $w_1xyxw_2yw_3$ also represents $T$. 

Representing each tree in a forest by using two letters (trees on one vertex $x$ and two vertices $x$ and $y$ can be represented by $xx$ and $xyxy$, respectively) and concatenating the obtained word-representants, we see that for any forest $F$ having at least two trees, $\mathcal{R}(F)= 2$. Indeed, having two letters in a word-represent for each tree guarantees that no pair of trees will be connected by an edge.

\subsubsection{Cycle Graphs}\label{cycle-sec} 
Another class of 2-representable graphs is {\em cycle graphs}. Note that a cyclic shift of a word-representant may not represent the same graph, as is the case with, say, the word 112. However, if a word-representant is uniform, a cyclic shift does represent the same graph, which is recorded in the following proposition. 

\begin{proposition}[\cite{KP08}]\label{cyclic-shift} Let $w = uv$ be a $k$-uniform word representing a graph $G$, where $u$ and $v$ are two, possibly empty, words. Then the word $w' = vu$ also represents $G$. \end{proposition}

Now, to represent a cycle graph $C_n$ on $n$ vertices, one can first represent the {\em path graph} $P_n$ on $n$ vertices using the technique to represent trees, then make a 1-letter cyclic shift still representing $P_n$ by Proposition~\ref{cyclic-shift}, and swap the first two letters. This idea is demonstrated for the graph in Figure~\ref{cycle-rep} as follows. The steps in representing the path graph $P_6$ obtained by removing the edge 16 from $C_6$ are
$$1212\rightarrow 121323 \rightarrow 12132434 \rightarrow 1213243545 \rightarrow 121324354656.$$
The 1-letter cyclic shift gives the word $612132435465$ still representing $P_6$ by Proposition~\ref{cyclic-shift}, and swapping the first two letters gives the sought representation of $C_6$: $162132435465$.  

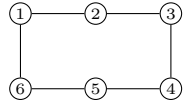
\begin{figure}[h]
\begin{center}
\begin{tikzpicture}[node distance=1cm,auto,main node/.style={circle,draw,inner sep=1pt,minimum size=2pt}]

\node[main node] (1) {{\tiny 1}};
\node[main node] (2) [right of=1] {{\tiny 2}};
\node[main node] (3) [right of=2] {{\tiny 3}};
\node[main node] (4) [below of=3] {{\tiny 4}};
\node[main node] (5) [left of=4] {{\tiny 5}};
\node[main node] (6) [left of=5] {{\tiny 6}};

\path
(1) edge (2)
(2) edge (3)
(3) edge (4)
(4) edge (5)
(5) edge (6)
(6) edge (1);

\end{tikzpicture}

\end{center}
\vspace{-5mm}
\caption{Cycle graph $C_6$}\label{cycle-rep}
\end{figure}

\subsubsection{Ladder Graphs} 

The ladder graph $L_n$ with $2n$ vertices, labeled $1,\ldots, n, 1',\ldots,n'$, and $3n -2$ edges is constructed following the pattern for $n=4$ presented in Figure~\ref{ladder-pic}. An inductive argument given in \cite{K13} shows that for $n\geq 2$, $\mathcal{R}(L_n)=2$. Table~\ref{ex-2-repr-Ln} gives 2-representations of $L_n$ for $n=1,2,3,4$. 

\begin{figure}[h]
\begin{center}
\begin{tikzpicture}[node distance=1cm,auto,main node/.style={circle,draw,inner sep=2pt}]

\node[main node] (1) {{\tiny 1}};
\node[main node] (2) [right of=1] {{\tiny 2}};
\node[main node] (3) [right of=2] {{\tiny 3}};
\node[main node] (4) [below of=3] {{\tiny $3'$}};
\node[main node] (5) [left of=4] {{\tiny $2'$}};
\node[main node] (6) [left of=5] {{\tiny $1'$}};
\node[main node] (7) [right of=3] {{\tiny 4}};
\node[main node] (8) [right of=4] {{\tiny $4'$}};

\path
(2) edge (5)
(1) edge (2)
(2) edge (3)
(3) edge (4)
(3) edge (7)
(4) edge (5)
(4) edge (8)
(7) edge (8)
(5) edge (6)
(6) edge (1);

\end{tikzpicture}

\end{center}
\vspace{-5mm}
\caption{The ladder graph $L_4$}\label{ladder-pic}
\end{figure}
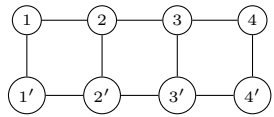

\begin{table}
\begin{center}
\begin{tabular}{c|c}
$n$ & $2$-representation of the ladder graph $L_n$ \\
\hline
1 &  $1{\bf 1'1}1'$\\
\hline
2 &  $1'21 {\bf 2'2} 1'2'1$\\
\hline
3 & $12'1'32 {\bf 3'3} 2'3'121'$\\ 
\hline
4 & $1'213'2'43 {\bf 4'4} 3'4'231'2'1$\\
\hline
\end{tabular}
\caption{$2$-representations of the ladder\index{ladder graph} graph $L_n$ for $n=1,2,3,4$}\label{ex-2-repr-Ln}
\end{center}
\end{table}

\subsubsection{Circle Graphs} 
A {\em circle graph} is the {\em intersection graph} of a set of chords of a circle. That is, it is an unoriented graph whose vertices can be associated with chords of a circle such that two vertices are adjacent if and only if the corresponding chords cross each other. See Figure~\ref{circle-graph-pic} for an example of a circle graph on four vertices and its associated chords. 

The following theorem provides a complete characterization of $\mathcal{R}_2$.

\begin{figure}[h]
\begin{center}

\begin{tabular}{ccc}

%
%
%
%
%

\begin{tikzpicture}[node distance=1cm,auto,main node/.style={circle,draw,inner sep=1pt,minimum size=2pt}]

\node[main node] (1) {{\tiny 3}};
\node[main node] (2) [below left of=1] {{\tiny 1}};
\node[main node] (3) [below right of=1] {{\tiny 2}};
\node[main node] (4) [below right of=2] {{\tiny 4}};

\path
(1) edge (2)
(1) edge (3);

\end{tikzpicture}

&

\ \ \ \ \ \ \

&

\begin{tikzpicture}[node distance=1cm,auto,main node/.style={circle,draw,inner sep=1pt,minimum size=2pt}]

 \node[draw,circle, minimum size=2cm,]  (a) {};
\foreach \x in {0,...,7}
\node  at (45*\x:1cm) [fill=white,circle, minimum size=0.015cm]{};
 
 \node  at (a.0) [draw,circle,inner sep=1pt, minimum size=2pt](a0) {\tiny{$3$}};
\node  at (a.45) [draw,circle,inner sep=1pt, minimum size=2pt](a1) {\tiny{$2$}};

  \node  at (a.90) [draw,circle,inner sep=1pt, minimum size=2pt](a2) {\tiny{$1$}};

\node  at (a.135) [draw,circle,inner sep=1pt, minimum size=2pt](a3) {\tiny{$4$}};

 \node  at (a.180) [draw,circle,inner sep=1pt, minimum size=2pt](a4) {\tiny{$4$}};

  \node  at (a.225) [draw,circle,inner sep=1pt, minimum size=2pt](a5) {\tiny{$3$}};

   \node  at (a.270) [draw,circle,inner sep=1pt, minimum size=2pt](a6) {\tiny{$1$}};

    \node  at (a.315) [draw,circle,inner sep=1pt, minimum size=2pt](a7) {\tiny{$2$}};
    \draw (a3)--(a4);

   \draw (a0)--(a5);
   \draw (a6)--(a2);
    \draw (a1)--(a7);
\end{tikzpicture}

\end{tabular}

\end{center}
\vspace{-5mm}
\caption{A circle graph on four vertices and its associated chords}\label{circle-graph-pic}
\end{figure}
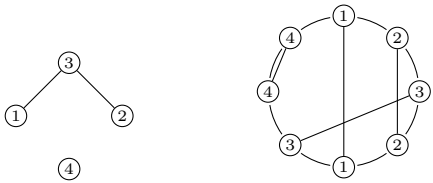

\begin{theorem}\label{circle-gr-thm} We have 
$$\mathcal{R}_2=\{ G : G \mbox{ is a circle graph different from a complete graph}\}.$$ \end{theorem}

\begin{proof} Given a circle graph $G$, consider its representation on a circle by intersecting chords. Starting from any chord's endpoint, go through all the endpoints in clock-wise direction recording chords' labels. The obtained word $w$ is 2-uniform and it has the property that a pair of letter $x$ and $y$ alternate in $w$ if and only if the pair of chords labeled by $x$ and $y$ intersect, which happens if and only if the vertex $x$ is connected to the vertex $y$ in $G$. For the graph in Figure~\ref{circle-graph-pic}, the chords' labels can be read starting from the lower 1 as 13441232, which is a 2-uniform word representing the graph. Thus, $G$ is a circle graph if and only if $G\in \mathcal{R}_2$ with the only exception if $G$ is a complete graph, in which case $G\in \mathcal{R}_1$. \end{proof}

\subsection{Graphs with Representation Number 3}

Unlike the case of graphs with representation number 2, no characterization of graphs with representation number 3 is know. However, there is a number of interesting results on this class of graphs to be discussed next. 

\subsubsection{The Petersen Graph}\label{Petersen-sec}

In 2010, Alexander Konovalov and Steven Linton not only showed that the Petersen graph in Figure~\ref{Petersen-graph} is not 2-representable, but also provided two {\em non-equivalent} (up to renaming letters or a cyclic shift) 3-representations of it:
\begin{itemize}
\item 1387296(10)7493541283(10)7685(10)194562 and 
\item 134(10)58679(10)273412835(10)6819726495.
\end{itemize}

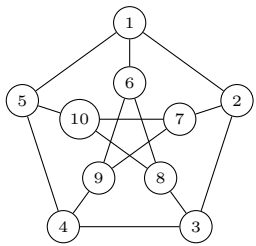
\begin{figure}
\begin{center}
\begin{tiny}
\begin{tikzpicture}[every node/.style={draw,circle}]
  \graph [clockwise,math nodes] {     
    subgraph C [V={ {1}, {2}, {3}, {4}, {5} }, name=A, radius=1.5cm]; 
    subgraph N [V={ {6}, {7}, {8}, {9}, {10} }, name=B, radius=0.7cm];
    \foreach \i [evaluate={\j=int(mod(\i+1,5)+1);}] in {1,...,5}{
      A \i -- B \i;  
      B \i -- B \j;
    }
  }; 
\end{tikzpicture}
\end{tiny}
\end{center}
\vspace{-5mm}
\caption{The Petersen graph}\label{Petersen-graph}
\end{figure}

The fact that the Petersen graph does not belong to $\mathcal{R}_2$ is also justified by the following theorem. 

\begin{theorem}[\cite{HKP10}] Petersen's graph is {\em not} $2$-representable. \end{theorem}

\begin{proof} Suppose that the graph is 2-representable and $w$ is a 2-uniform word representing it. Let $x$ be a letter in $w$ such that there is a minimal number of letters between the two occurrences
of $x$. Since Petersen's graph is regular of degree 3,  it is not difficult to see that there must be exactly three letters, which are all different, between the $x$s (having more letters between $x$s would lead to having two equal letters there, contradicting the choice of $x$). 

By symmetry, we can assume that $x=1$, and by
Proposition~\ref{cyclic-shift} we can assume that $w$ starts with 1. So,
the letters 2, 5 and 6 are between the two 1s, and because of symmetry,
the fact that Petersen's graph is {\em edge-transitive} (that is, each of
its edges can be made ``internal''), and taking into account that
the vertices 2, 5 and 6 are pairwise non-adjacent, we can assume that
$w=12561w_16w_25w_32w_4$ where the $w_i$s are some, possibly empty words for
$i\in\{1,2,3,4\}$. To alternate with 6 but not to alternate with 5, the letter 8 must occur in 
$w_1$ and $w_2$. Also, to alternate with 2 but
not to alternate with 5, the letter 3 must occur in $w_3$ and $w_4$. But
then 8833 is a subsequence in $w$, and thus 8 and 3 must be non-adjacent
in the graph, a contradiction. \end{proof}

\subsubsection{Prisms}\label{prisms-sec}

A {\em prism} $\mbox{Pr}_n$ is a graph consisting of two cycles $12\cdots n$ and $1'2'\cdots n'$, where $n\geq 3$, connected by the edges $ii'$ for $i=1,\ldots ,n$. In particular, the 3-dimensional cube to the right in Figure~\ref{prisms-3-4} is the prism $\mbox{Pr}_4$.

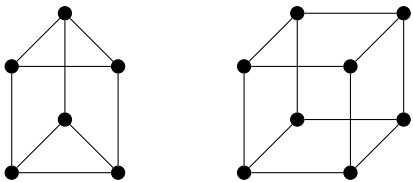
\begin{figure}
\begin{center}
\begin{tabular}{ccc}
\begin{tikzpicture}[node distance=1cm,auto,main node/.style={fill,circle,draw,inner sep=0pt,minimum size=5pt}]

\node[main node] (1) {};
\node[main node] (2) [below left of=1] {};
\node[main node] (3) [below right of=1] {};
\node[main node] (4) [below right of=2] {};
\node[main node] (5) [below left of=4] {};
\node[main node] (6) [below right of=4] {};

\path
(1) edge (2)
(1) edge (4)
(1) edge (3)
(2) edge (3)
(2) edge (5)
(3) edge (6)
(4) edge (5)
(4) edge (6)
(5) edge (6);

\end{tikzpicture}

& 

\ \ \ \ \ \ 

&

\begin{tikzpicture}[node distance=1cm,auto,main node/.style={fill,circle,draw,inner sep=0pt,minimum size=5pt}]

\node[main node] (1) {};
\node[main node] (2) [below left of=1] {};
\node[main node] (3) [below right of=1] {};
\node[main node] (4) [below right of=2] {};
\node[main node] (5) [below left of=4] {};
\node[main node] (6) [below right of=4] {};
\node[main node] (7) [above right of=3] {};
\node[main node] (8) [above right of=6] {};

\path
(1) edge (2)
(1) edge (4)
(2) edge (3)
(2) edge (5)
(7) edge (1)
(7) edge (3)
(7) edge (8)
(3) edge (6)
(8) edge (6)
(8) edge (4)
(4) edge (5)
(5) edge (6);

\end{tikzpicture}
\end{tabular}
\end{center}
\vspace{-5mm}
\caption{Prisms $\mbox{Pr}_3$ and $\mbox{Pr}_4$}\label{prisms-3-4}
\end{figure}

\begin{theorem}[\cite{KP08}]\label{every-prism} {\em Every} prism $\mbox{Pr}_n$ is $3$-representable. \end{theorem} 

The fact that the triangular prism $\mbox{Pr}_3$ is not 2-representable was shown in \cite{KP08}. The following more general result holds.

\begin{theorem}[\cite{K13}]\label{none-of-prisms} {\em None} of prisms $\mbox{Pr}_n$ is $2$-representable. \end{theorem}
 
Theorems~\ref{every-prism} and~\ref{none-of-prisms} show that $\mbox{Pr}_n\in \mathcal{R}_3$ for any $n\geq 3$.
 
\subsubsection{Colorability of Graphs in $\mathcal{R}_3$}
Theorem~\ref{crown-thm-main} below shows that  $\mathcal{R}_3$ does not even include $2$-colorable graphs, and thus any class of $c$-colorable graphs for $c\geq 3$. Indeed, any $c$-colorable non-3-representable graph can be extended to a $(c+1)$-colorable graph by adding an apex (all-adjacent vertex), which is still non-3-representable using the hereditary nature of word-representability (see Remark~\ref{hereditary}). 

A natural question to ask here is: Is $\mathcal{R}_3$ properly included in a class of $c$-colorable graphs for a constant $c$? A simple argument of replacing a vertex in the 3-representable triangular prism $\mbox{Pr}_3$ by a complete graph of certain size led to the following theorem.

\begin{theorem}[\cite{K13}]\label{thm-R3-c-colorable} The class  $\mathcal{R}_3$ is not included in a class of $c$-colorable graphs for some constant~$c$. \end{theorem}

\subsubsection{Subdivisions of Graphs}\label{subdiv-graphs-3}

The following theorem gives a useful tool for constructing $3$-representable graphs, that is, graphs with representation number at most $3$.

\begin{theorem}[\cite{KP08}]\label{add-path-to-graph}
Let $G=(V,E)$ be a $3$-representable graph and $x,y\in V$. Denote by
$H$ the graph obtained from $G$ by adding to it a path of length
at least $3$ connecting $x$ and $y$. Then $H$ is also
$3$-representable.
\end{theorem}

\begin{definition}\label{def-subdivision} A {\it subdivision} of a graph $G$ is a graph obtained from $G$ by replacing each edge $xy$ in $G$ by a {\em simple} path (that is, a path without self-intersection) from $x$ to $y$. A subdivision is called a {\em $k$-subdivision} if each of these paths is of length at least $k$.\end{definition}

\begin{definition}\label{edge-contraction-undirected} An {\em edge contraction} is an operation which removes an edge from a graph while gluing the two vertices it used to connect. An unoriented graph $G$ is a {\em minor} of another unoriented graph $H$ if a graph isomorphic to $G$ can be obtained from $H$ by contracting some edges, deleting some edges, and deleting some isolated vertices. \end{definition}

\begin{theorem}[\cite{KP08}]\label{minor-3-word-repr} For every graph $G$ there are infinitely many $3$-representable graphs $H$
that contain $G$ as a minor. Such a graph $H$ can be obtained from $G$ by subdividing {\em each} edge into {\em any} number of, but {\em at least} three edges. 
 \end{theorem}

Note that $H$ in Theorem~\ref{minor-3-word-repr} does not have to be a $k$-subdivision for some $k$, that it, edges of $G$ can be subdivided into different number (at least 3) of edges. In either case, the 3-subdivision of any graph $G$ is always 3-representable. Also, it follows from Theorem~\ref{add-path-to-graph} and the proof of Theorem~\ref{minor-3-word-repr} in \cite{KP08} that a graph obtained from an edgeless graph by inserting simple paths of length at least $3$ between (some) pairs of vertices of the graph is $3$-representable.

Finally, note that subdividing {\em each} edge in any graph into {\em exactly} two edges gives a bipartite graph, which is word-representable by Theorem~\ref{Kit-Sei-thm}  (see the discussion in Section~\ref{crown-graph-sec} on why a bipartite graph is word-representable).

\subsection{Graphs with High Representation Number}

In Theorem~\ref{main-charact} below we will see that the upper bound on a shortest word-representant for a graph $G$ on $n$ vertices is essentially $2n^2$, that is, one needs at most $2n$ copies of each letter to represent $G$. Next, we consider two classes of graphs that require essentially $n/2$ copies of each letter to be represented, and these are the longest known shortest word-representants.

\subsubsection{Crown Graphs}\label{crown-graph-sec}
\begin{definition}  A {\em crown graph} (also known as a {\em cocktail party graph}) $H_{n,n}$ is obtained from the complete bipartite graph $K_{n,n}$ by removing a {\em perfect matching}. That is, $H_{n,n}$ is obtained from $K_{n,n}$ by removing  $n$ edges such that {\em each} vertex was incident to {\em exactly one} removed edge. \end{definition}

\begin{figure}
\begin{center}

\begin{tabular}{ccccc}

\begin{tikzpicture}[node distance=1cm,auto,main node/.style={circle,draw,inner sep=1pt,minimum size=5pt}]

\node[main node] (1) {1};
\node[main node] (2) [below of=1] {$1'$};
\node (a) [below of=2,yshift=2mm] {$H_{1,1}$};


\end{tikzpicture}

&

\ \ \ 

&

\begin{tikzpicture}[node distance=1cm,auto,main node/.style={circle,draw,inner sep=1pt,minimum size=5pt}]

\node[main node] (1) {1};
\node[main node] (2) [below of=1] {$1'$};
\node[main node] (3) [right of=1] {2};
\node[main node] (4) [below of=3] {$2'$};
\node (a) [below of=2,yshift=2mm,xshift=5mm] {$H_{2,2}$};

\path
(1) edge node  {} (4);

\path
(3) edge node  {} (2);

\end{tikzpicture}

&

\ \ \ 

&

\begin{tikzpicture}[node distance=1cm,auto,main node/.style={circle,draw,inner sep=1pt,minimum size=5pt}]

\node[main node] (1) {1};
\node[main node] (2) [below of=1] {$1'$};
\node[main node] (3) [right of=1] {2};
\node[main node] (4) [below of=3] {$2'$};
\node[main node] (5) [right of=3] {3};
\node[main node] (6) [below of=5] {$3'$};
\node (a) [below of=4,yshift=2mm] {$H_{3,3}$};

\path
(1) edge node  {} (4)
     edge node  {} (6);

\path
(3) edge node  {} (2)
     edge node  {} (6);

\path
(5) edge node  {} (2)
     edge node  {} (4);

\end{tikzpicture}

\vspace{-4mm}

\ 

\end{tabular}

\end{center}
\vspace{-2mm}
\caption{Crown graphs}\label{crown-graphs-fig}
\end{figure}
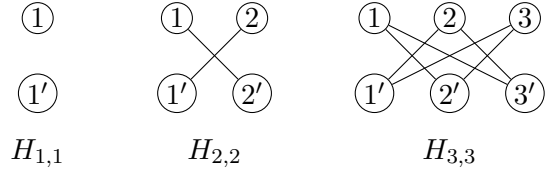

See Figure~\ref{crown-graphs-fig} for examples of crown graphs. 

By Theorem~\ref{Kit-Sei-thm} below, $H_{n,n}$ can be represented by a concatenation of permutations, because $H_{n,n}$ is a comparability graph defined in Section~\ref{perm-rep-graphs-sec} (to see this, just orient all edges from one part to the other). In fact, $H_{n,n}$ is known to require $n$ permutations to be represented. However, can we provide a shorter representation for $H_{n,n}$? It turns out that we can, to be discussed next, but such representations are still long (linear in $n$). 

Note that $H_{1,1}\in\mathcal{R}_2$ by Section~\ref{empty-sec}. Further, $H_{2,2}\neq K_4$, the complete graph on $4$ vertices, and thus $H_{2,2}\in\mathcal{R}_2$ because it can be 2-represented by $121'2'212'1'$. Also, $H_{3,3}=C_6\in\mathcal{R}_2$ by Section~\ref{cycle-sec}. Finally, $H_{4,4}=\mbox{Pr}_4\in\mathcal{R}_3$ by Section~\ref{prisms-sec}. The following theorem gives the representation number $\mathcal{R}(H_{n,n})$ in the remaining cases.

\begin{theorem}[\cite{G16}]\label{crown-thm-main} If $n\geq 5$ then the representation number of $H_{n,n}$ is $\lceil n/2 \rceil$ (that is, one needs $\lceil n/2 \rceil$ copies of each letter to represent $H_{n,n}$, but {\em not fewer}).\end{theorem}

\subsubsection{Crown Graphs with an Apex}

The graph $G_n$ is obtained from a crown graph $H_{n,n}$ by adding an apex (all-adjacent vertex). See Figure~\ref{G3-pic} for the graph $G_3$. 

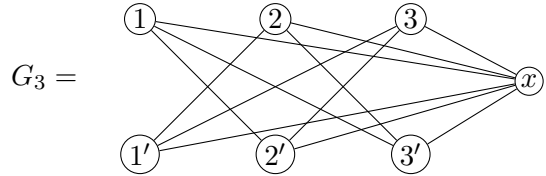
\begin{figure}
\begin{center}
\begin{tikzpicture}[node distance=1.8cm,auto,main node/.style={circle,draw,inner sep=1pt,minimum size=5pt}]

\node[main node] (1) {1};
\node[main node] (2) [below of=1] {$1'$};
\node[main node] (3) [right of=1] {2};
\node[main node] (4) [below of=3] {$2'$};
\node[main node] (5) [right of=3] {3};
\node[main node] (6) [below of=5] {$3'$};
\node[main node] (7) [above right of=6,yshift=-3mm,xshift=3mm] {$x$};
\node (a) [below left of=1,yshift=5mm] {$G_{3}=$};

\path
(1) edge node  {} (4)
     edge node  {} (6);

\path
(3) edge node  {} (2)
     edge node  {} (6);

\path
(5) edge node  {} (2)
     edge node  {} (4);

\path
(7) edge node  {} (1)
     edge node  {} (2)
     edge node  {} (3)
     edge node  {} (4)
     edge node  {} (5)
     edge node  {} (6);

\end{tikzpicture}
\end{center}
\vspace{-5mm}
\caption{Graph $G_3$}\label{G3-pic}
\end{figure}

It turns out that $G_n$ is the {\em worst} known word-representable graph in the sense that it requires the maximum number of copies of each letter to be represented, as recorded in the following theorem.  

\begin{theorem}[\cite{KP08}] The representation number of $G_n$ is $\lfloor n/2 \rfloor$.\end{theorem}

It is unknown whether there exist graphs on $n$ vertices with representation number between $\lfloor n/2 \rfloor$ and essentially $2n$ (given by Theorem~\ref{main-length}).

\section{Permutationally Representable Graphs and their Significance}\label{perm-rep-graphs-sec}

An orientation of a graph is {\em transitive} if presence of edges $u\rightarrow v$ and $v\rightarrow z$ implies presence of the edge $u\rightarrow z$.  An unoriented graph is a {\em comparability graph} if it admits a transitive orientation. It is well known \cite[Section 3.5.1]{KL15}, and is not difficult to show that the smallest non-comparability graph is the cycle graph $C_5$.

\begin{definition} A graph $G=(V,E)$ is {\em permutationally representable} if it can be represented by a word of the form $p_1\cdots p_k$ where $p_i$ is a permutation. We say that $G$ is {\em permutationally {\em k}-representable}.
\end{definition}

For example, the graph in Figure~\ref{wrg-ex} is permutationally representable, which is justified by the concatenation of two permutations 21342341.

The following theorem is an easy corollary of the fact that any partially ordered set can be represented as intersection of linear orders.

\begin{theorem}[\cite{KS08}]\label{Kit-Sei-thm} A graph is permutationally representable if and only if it is a comparability graph.\end{theorem}

Next, consider a schematic representation of the graph $G$ in Figure~\ref{G-H} obtained from a graph $H$ by adding an all-adjacent vertex (apex). The following theorem holds.

\begin{theorem}[\cite{KP08}]\label{perm-versus-wr} The graph $G$ is word-representable if and only if the graph $H$ is permutationally representable.\end{theorem}

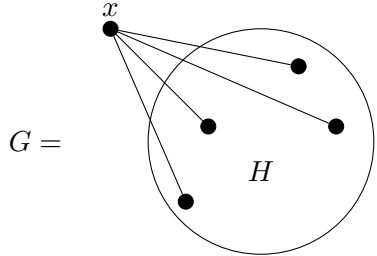
\begin{figure}
\begin{center}

\begin{tikzpicture}[node distance=1cm,auto,main node/.style={fill,circle,draw,inner sep=3pt,minimum size=5pt}]

\draw (2,2) circle (1.5cm);
    
    \node[draw,fill,circle,inner sep=2pt,minimum size=2pt] at (1, 1.2)   (1) { };
    \node[draw,fill,circle,inner sep=2pt,minimum size=2pt] at (1.3, 2.2)   (2) { };
\node[draw,fill,circle,inner sep=2pt,minimum size=2pt] at (2.5, 3)   (3) {};
    \node[draw,fill,circle,inner sep=2pt,minimum size=2pt] at (3, 2.2)   (4) {};
 \node[draw,fill,circle,inner sep=2pt,minimum size=2pt] at (0, 3.5)   (5) {};
 
  \node at (0, 3.75)   (6) {$x$}; 
  \node at (2, 1.6)   (7) {$H$};
 \node at (-1, 2)   (7) {$G=$};
   
    \draw[] (5) -- (1);
    \draw[] (5) -- (2);
   \draw[] (5) -- (3);
   \draw[] (5) -- (4);

\end{tikzpicture}
\end{center}
\vspace{-5mm}
\caption{$G$ is obtained from $H$ by adding an apex}\label{G-H}
\end{figure}

A {\em wheel graph} $W_n$ is the graph obtained from a cycle graph $C_n$ by adding an apex. It is easy to see that none of cycle graphs $C_{2n+1}$, for $n\geq 2$, is a comparability graph, and thus none of wheel graphs $W_{2n+1}$, for $n\geq 2$ is word-representable. In fact, $W_5$ is the smallest example of a non-word-representable graph (the only one on 6 vertices). Section~\ref{non-w-r-sec}  discusses other examples of non-word-representable graphs. 

As a direct corollary to Theorem~\ref{perm-versus-wr}, we have the following important result revealing the structure of neighbourhoods of vertices in a word-representable graph.

\begin{theorem}[\cite{KP08}]\label{neighbourhood} If a graph $G$ is word-representable then the neighbourhood of each vertex in $G$ is permutationally representable (is a comparability graph by Theorem~\ref{Kit-Sei-thm}). \end{theorem}

The converse to Theorem~\ref{neighbourhood} is {\em not} true as demonstrated by the counterexamples in Figure~\ref{2-counterex} taken from \cite{HKP10} and \cite{CKL17}, respectively.

\begin{figure}
\begin{center}
\begin{tabular}{ccc}

\begin{tikzpicture}[node distance=1cm,auto,main node/.style={fill,circle,draw,inner sep=0pt,minimum size=5pt}]

\node[main node] (1) {};
\node[main node] (2) [above left of=1,yshift=-0.3cm,xshift=0.3cm] {};
\node[main node] (3) [above right of=1,yshift=-0.3cm,xshift=-0.3cm] {};
\node[main node] (4) [below of=1,yshift=0.5cm] {};
\node[main node] (5) [above right of=2,xshift=-0.3cm,yshift=0.5cm] {};
\node[main node] (6) [below right of=4,xshift=1cm,yshift=0.3cm] {};
\node[main node] (7) [below left of=4,xshift=-1cm,yshift=0.3cm] {};
\node (8) [left of=5,xshift=-0.5cm] {co-($T_2$)};

\path
(5) edge (7) 
(5) edge (2) 
(5) edge (3) 
(5) edge (6) 
(2) edge (3) 
(2) edge (1) 
(2) edge (4) 
(2) edge (7) 
(1) edge (3) 
(1) edge (4) 
(6) edge (3) 
(6) edge (4) 
(6) edge (7)
(4) edge (7)
(4) edge (3);

\end{tikzpicture}

&
\ \ 
&

\begin{tikzpicture}[node distance=1cm,auto,main node/.style={fill,circle,draw,inner sep=0pt,minimum size=5pt}]
\node[main node] (1) {};
\node[main node] (2) [below left of=1] {};
\node[main node] (3) [below right of=1] {};
\node[main node] (4) [left of=2] {};
\node[main node] (5) [right of=3] {};
\node[main node] (6) [below right of=2] {};
\node[main node] (7) [below of=6] {};

\path
(4) edge (5) 
(1) edge (5) 
(1) edge (4) 
(1) edge (2) 
(1) edge (3) 
(6) edge (2) 
(6) edge (3) 
(6) edge (7) 
(7) edge (4) 
(7) edge (5);

\end{tikzpicture}

\end{tabular}
\end{center}
\vspace{-5mm}
\caption{Non-word-representable graphs in which each neighbourhood is permutationally representable}\label{2-counterex}
\end{figure}
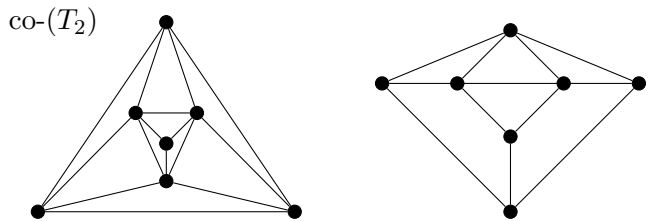

A {\em clique} in an unoriented graph is a subset of pairwise adjacent vertices. A {\em maximum clique} is a clique of the maximum size. Given a graph $G$, the {\em Maximum Clique problem} is to find a maximum clique in $G$. It is well known that the Maximum Clique problem is NP-complete. However, this problem is polynomially solvable for word-representable graphs, which is a corollary of Theorem~\ref{neighbourhood} and is discussed next.

\begin{theorem}[\cite{HKP11,HKP16}]\label{Max-Click-thm} The Maximum Clique problem is polynomially solvable on word-representable graphs.\end{theorem}

\begin{proof} Each neighbourhood of a word-representable graph $G$ is a comparability graph by Theorem~\ref{neighbourhood}. It is known that the Maximum Clique problem is solvable on comparability graphs in polynomial time. Thus the problem is solvable on $G$ in polynomial time, since any maximum clique belongs to the neighbourhood of a vertex including the vertex itself.\end{proof}

\section{Graphs Representable by Pattern Avoiding Words}\label{sec3}

It is a very popular area of research to study patterns in words and permutations\footnote{The patterns considered in this section are ordered, and their study comes from Algebraic Combinatorics. There are a few results on word-representable graphs and (unordered) patterns studied in Combinatorics on Words, namely on squares and cubes in words, that are not presented in this paper, but can be found in~\cite[Section 7.1.3]{KL15}. One of the results says that for any word-representable graph, there exists a cube-free word representing it.}. The book~\cite{K11} provides a comprehensive introduction to the field. A {\em pattern} is a word containing each letter in $\{1,2,\ldots,k\}$  at least once for some $k$. A pattern $\tau=\tau_1\tau_2\cdots \tau_{m}$ {\em occurs} in a word $w=w_1w_2\cdots w_n$ if there exist $1\leq i_1< i_2< \cdots< i_m\leq n$ such that $\tau_1\tau_2\cdots \tau_{m}$ is {\em order-isomorphic} to $w_{i_1}w_{i_2}\cdots w_{i_m}$. We say that $w$ {\em avoids} $\tau$ if $w$ contains no occurrences of $\tau$. For example, the word 42316 contains several occurrences of the pattern 213 (all ending with 6), e.g. the subsequences 426, 416 and 316.  

As a particular case of a more general program of research suggested by the author during his plenary talk at the international Permutation Patterns Conference at the East Tennessee State University, Johnson City in 2014, one can consider the following direction (see \cite[Section 7.8]{KL15}). Given a set of words avoiding a pattern, or a set of patterns, which class of graphs do these words represent?

As a trivial example, consider the class of graphs defined by words avoiding the pattern 21. Clearly, any 21-avoiding word is of the form $$w=11\cdots 122\cdots 2 \cdots nn\cdots n.$$ If a letter $x$ occurs {\em at least twice} in $w$ then the respective vertex is isolated. The letters occurring {\em exactly once} form a {\em clique} (are connected to each other). Thus, 21-avoiding words describe graphs formed by a clique and an independent set.

Two papers, \cite{GKZ16} and \cite{M16}, are dedicated to this research direction and will be summarised in this section. So far, apart from Theorem~\ref{thm-avoid-gen} and Corollary~\ref{cor-avoid-gen} below, only 132-avoiding and 123-avoiding words were studied from word-representability point of view. The results of these studies are summarized in Figure~\ref{132-123-sum}, which is taken from \cite{M16}. In that figure, and more generally in this section, we slightly abuse the notation and call graphs representable by $\tau$-avoiding words {\em $\tau$-representable}.  

\begin{figure}
\begin{center}\includegraphics[height=4cm]{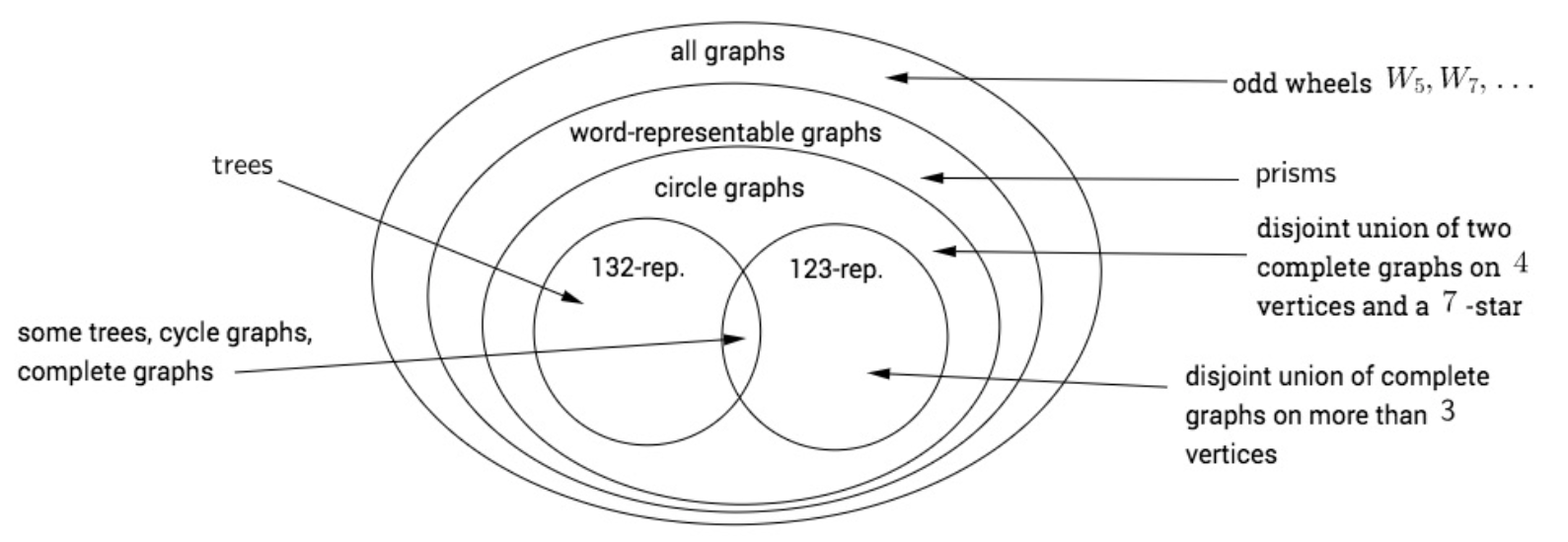} \end{center}
\vspace{-5mm}
\caption{Relations between graph classes taken from \cite{M16}}\label{132-123-sum}
\end{figure}

We note that unlike the case of word-representability without extra restrictions, labeling of graphs {\em does matter} in the case of pattern avoiding representations. For example, the 132-avoiding word 4321234 represents the graph to the left in  Figure~\ref{labeling-matters}, while {\em no} 132-avoiding word represents the other graph in that figure. Indeed, {\em no two} letters out of 1, 2 and 3 can occur {\em once} in a word-representant or else the respective vertices would {\em not} form an {\em independent set}. Say, w.l.o.g. that 1 and 2 occur at {\em least twice}. But then we can find 1 and 2 on {\em both sides} of an occurrence of the letter 4, and the patten 132 is {\em inevitable}.

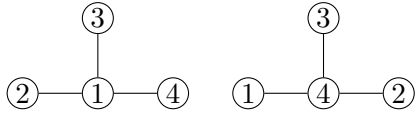
\begin{figure}
\begin{center}
\begin{tikzpicture}[node distance=1cm,auto,main node/.style={circle,draw,inner sep=1pt,minimum size=5pt}]

\node[main node] (1) {2};
\node[main node] (2) [right of=1] {1};
\node[main node] (3) [above of=2] {3};
\node[main node] (4) [right of=2] {4};
\node[main node] (5) [right of=4] {1};
\node[main node] (6) [right of=5] {4};
\node[main node] (7) [above of=6] {3};
\node[main node] (8) [right of=6] {2};

\path
(2) edge (1)
     edge (3)
     edge (4);

\path
(6) edge (5)
     edge (7)
     edge (8);

\end{tikzpicture}
\end{center}
\vspace{-5mm}
\caption{132-representable (left) and non-132-representable (right) labelings of the same graph}\label{labeling-matters}
\end{figure}

The following theorem has a great potential to be applicable to the study of $\tau$-representable graphs for $\tau$ of length 4 or more.

\begin{theorem}[\cite{M16}]\label{thm-avoid-gen}  Let $G$ be a word-representable graph, which can be represented by a word avoiding a pattern $\tau$ of length $k + 1$. Let $x$ be a vertex in $G$ such that its degree $d(x) \geq k$. Then, {\em any} word $w$ representing $G$ that avoids $\tau$ must contain no more than $k$ copies of $x$. \end{theorem}

\begin{proof} If there are {\em at least} $k+1$ occurrences of $x$ in $w$, we obtain a factor (i.e.\ consecutive subword) $xw_1x\cdots w_kx$, where $k$ neighbours of $x$ in $G$ occur in each $w_i$. But then $w$ contains {\em all} patterns of length $k+1$, in particular, $\tau$. Contradiction.  \end{proof}

\begin{corollary}[\cite{M16}]\label{cor-avoid-gen}  Let $w$ be a word-representant for a graph which avoids a pattern of length $k+1$. If some vertex $y$ adjacent to $x$ has degree {\em at least} $k$, then $x$ occurs {\em at most} $k+1$ times in $w$. \end{corollary}

\subsection{132-Representable Graphs}\label{132-rep-subsec} 

It was shown in \cite{GKZ16} that the minimum (with respect to the number of vertices) non-word-representable graph, the wheel graph $W_5$, is actually a minimum non-132-representable graph (we do not know if there exit other non-132-representable graphs on 6 vertices).

\begin{theorem}[\cite{GKZ16}]\label{thm3-2}
If a graph $G$ is $132$-representable, then there exists a $132$-avoiding word $w$ representing $G$ such that any letter in $w$ occurs at most twice.
\end{theorem}

Theorems~\ref{thm3-2} and~\ref{circle-gr-thm} give the following result.

\begin{theorem}[\cite{GKZ16}]\label{132-are-circle}
Every $132$-representable graph is a circle graph.
\end{theorem}

Thus, by Theorems~\ref{circle-gr-thm}, \ref{none-of-prisms} and~\ref{132-are-circle},  none of prisms $\mbox{Pr}_n$, $n\geq 3$, is 132-representable. A natural question is if there are circle graphs that are not 132-representable. 

\begin{theorem}[\cite{M16}]\label{two-K4} Not all circle graphs are $132$-representable. E.g. disjoint union of two complete graphs $K_4$ is a circle graph, but it is not $132$-representable. \end{theorem}

\begin{theorem}[\cite{GKZ16}]\label{132-trees-rep}
Any tree is $132$-representable.
\end{theorem}

Note that in the case of pattern avoiding representations of graphs, Theorem~\ref{equiv-thm} does not necessarily work, because extending a representation to a uniform representation may introduce an occurrences of the pattern(s) in question. For example, while any complete graph $K_n$ can be represented by the 132-avoiding word $n(n-1)\cdots 1$, it was shown in \cite{M16} that for $n\geq 3$ no 2-uniform 132-avoiding representation of $K_n$ exists. In either case, \cite{M16} shows that any tree can actually be represented by a 2-uniform word thus refining the statement of Theorem~\ref{132-trees-rep}. For another result on uniform 132-representation see Theorem~\ref{comp-123-132} below.

\begin{theorem}[\cite{GKZ16}]
Any cycle graph is $132$-representable.
\end{theorem}

\begin{proof} The cycle graph $C_n$ labeled by $1,2,\ldots,n$ in clockwise direction can be represented by the 132-avoiding word $$(n -1)n(n -2)(n -1)(n - 3)(n - 2)\cdots 45342312.$$
\end{proof}

\begin{theorem}[\cite{GKZ16}]\label{thm-K-n-132}
For $n\geq1$, a complete graph $K_n$ is $132$-representable. Moreover, for $n\geq 3$, there are $$2+C_{n-2}+\sum_{i=0}^{n}C_i$$ different $132$-representants for $K_n$, where $C_n=\frac{1}{n+1}{2n\choose n}$ is the $n$-th Catalan number. Finally, $K_1$ can be represented by a word of the form $11\cdots 1$ and $K_2$ by a word of the form $1212\cdots$ (of even  or odd length) or $2121\cdots$ (of even or odd length).
\end{theorem}

As a corollary to the proof of Theorem~\ref{thm-K-n-132}, \cite{GKZ16} shows that for $n\geq 3$, the length of any $132$-representant of  $K_n$ is either $n$, or $n+1$, or $n+2$, or $n+3$.

\subsection{123-Representable Graphs}

An analogue of Theorem~\ref{132-are-circle} holds for 123-representable graphs.

\begin{theorem}[\cite{M16}]
Any $123$-representable graph is a circle graph. 
\end{theorem}

\begin{figure}
\begin{center}
\begin{tikzpicture}[node distance=1cm,auto,main node/.style={fill,circle,draw,inner sep=0pt,minimum size=5pt}]

\node[main node] (1) {};
\node[main node] (2) [above left of=1] {};
\node[main node] (3) [above right of=1] {};
\node[main node] (4) [above of=1] {};
\node[main node] (5) [below left of=1] {};
\node[main node] (6) [below right of=1] {};
\node[main node] (7) [below of=1] {};

\path
(4) edge (7)
(1) edge (2)
(1) edge (3)
(1) edge (5)
(1) edge (6);

\end{tikzpicture}
\end{center}
\vspace{-5mm}
\caption{Star graph $K_{1,6}$}\label{star-K-1-6}
\end{figure}
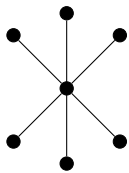

\begin{theorem}[\cite{M16}]
Any cycle graph is $123$-representable.
\end{theorem}

\begin{proof} The cycle graph $C_n$ labeled by $1,2,\ldots,n$ in clockwise direction can be represented by the 123-avoiding word $$n(n -1)n(n - 2)(n - 1)(n - 3)(n - 1) \cdots 23121.$$
\end{proof}

\begin{theorem}[\cite{M16}]\label{K-1-6}
The star graph $K_{1,6}$ in Figure~\ref{star-K-1-6} is not $123$-representable. 
\end{theorem}

It is easy to see that $K_{1,6}$ is a circle graph, and thus not all circle graphs are 123-representable by Theorem~\ref{K-1-6}. Also, by Theorem~\ref{K-1-6}, not all trees are 123-representable. 

Based on Theorems \ref{two-K4} and \ref{K-1-6}, it is easy to come up with a circle graph on 14 vertices that is neither 123- nor 132-representable (see \cite{M16}). 

As opposed to the situation with 132-representation discussed in Section~\ref{132-rep-subsec}, any complete graph $K_n$ can be represented by the 123-avoiding 2-uniform word $n(n-1)\cdots 1n(n-1)\cdots 1$ as observed in \cite{M16}. Also, it was shown in \cite{M16} that any path graph $P_n$ can be 123-represented by a 2-uniform word. We conclude with a general type theorem on uniform representation applicable to both 123- and 132-representations. 

\begin{theorem}[\cite{M16}]\label{comp-123-132} Let a pattern $\tau\in\{123,132\}$ and $G_1$, $G_2,\ldots, G_k$ be $\tau$-representable connected components of a graph $G$. Then $G$ is $\tau$-representable if and only if at most one of the connected components cannot be $\tau$-represented by a $2$-uniform word. \end{theorem}

\section{Semi-transitive Orientations as a Key Tool in the Theory of Word-Representable Graphs}\label{sec4}

Recall the definition of a transitive orientation at the beginning of Section~\ref{perm-rep-graphs-sec}.

A {\em shortcut} is an {\em acyclic non-transitively oriented} graph obtained from a directed cycle graph forming a directed cycle on at least four vertices by changing the orientation of one of the edges, and possibly by adding more directed edges connecting some of the vertices (while keeping the graph be acyclic and non-transitive). Thus, any shortcut  

\begin{itemize}

\item is {\em acyclic} (that it, there are {\em no directed cycles});

\item has {\em at least} 4 vertices;

\item has {\em exactly one} source (the vertex with no edges coming in), {\em exactly one} sink (the vertex with no edges coming out), and a {\em directed path} from the source to the sink that goes through {\em every} vertex in the graph;

\item has an edge connecting the source to the sink that we refer to as the {\em shortcutting edge};

\item is {\em not} transitive (that it, there exist vertices $u$, $v$ and $z$ such that $u\rightarrow v$ and $v\rightarrow z$ are edges, but there is {\em no} edge $u\rightarrow z$).

\end{itemize}

\begin{definition} An orientation of a graph is {\em semi-transitive} if it is {\em acyclic} and 
{\em shortcut-free}.\end{definition}

It is easy to see from definitions that {\em any} transitive orientation is necessary  semi-transitive. The converse is {\em not} true, e.g.\ the following schematic semi-transitively oriented graph is {\em not} transitively oriented: 

\vspace{-3mm}

\begin{center}

\begin{tikzpicture}[->,>=stealth', shorten >=1pt, node distance=2.5cm,auto,main node/.style={fill,circle,draw,inner sep=0pt,minimum size=5pt}]

\node[main node] (1) {};
\node (a) [right of=1] {transitively oriented};
\node[main node] (2) [right of=a] {};
\node (b) [right of=2] {transitively oriented};
\node[main node] (3) [right of=b] {};

\path
(1) edge [bend right=20] node  {} (2)
      edge [bend left=20] node  {} (2);

\path
(2) edge [bend right=20] node  {} (3)
      edge [bend left=20] node  {} (3);

\end{tikzpicture}

\vspace{-3mm}

\end{center}

Thus semi-transitive orientations generalize transitive orientations.

A way to check if a given oriented graph $G$ is semi-transitively oriented is as follows. First  check that $G$ is acyclic; if not, the orientation is not semi-transitive. Next, for a directed edge from a vertex $x$ to a vertex $y$, consider {\em each} directed path $P$ having at least three edges without repeated vertices from $x$ to $y$, and check that the subgraph of $G$ induced by $P$ is transitive. If such non-transitive subgraph is found, the orientation is not semi-transitive. This procedure needs to be applied to each edge in $G$, and if no non-transitivity is discovered, $G$'s orientation is semi-transitive. 

As we will see in Theorem~\ref{main-charact}, finding a semi-transitive orientation is equivalent to recognising whether a given graph is word-representable, and this is an NP-hard problem (see Theorem~\ref{recogn-word-repr-NP-complete}). Thus, there is no efficient way to construct a semi-transitive orientation in general, and such a construction would rely on an exhaustive search orienting edges one by one, and thus branching the process. Having said that, there are several situations in which branching is not required. For example, the orientation of the partially oriented triangle below can be completed uniquely to avoid a cycle:

 \begin{center}
\begin{tabular}{ccc}

\begin{tikzpicture}[node distance=1cm,auto,main node/.style={fill,circle,draw,inner sep=0pt,minimum size=5pt}]

\node[main node] (1) {};
\node[main node] (2) [below left of=1] {};
\node[main node] (3) [below right of=1] {};

\path
(2) edge node  {} (3);

\path
(2) [->,>=stealth', shorten >=1pt] edge node  {} (1);

\path     
(1) [->,>=stealth', shorten >=1pt] edge node  {} (3);     

\end{tikzpicture}

&

\begin{tikzpicture}[node distance=0.4cm,auto,main node/.style={inner sep=0pt,minimum size=5pt}]

\node [main node] (1) {\ $\Rightarrow$ \ };

\node[main node] (2) [below of=1] {};

\end{tikzpicture}

&

\begin{tikzpicture}[node distance=1cm,auto,main node/.style={fill,circle,draw,inner sep=0pt,minimum size=5pt}]

\node[main node] (1) {};
\node[main node] (2) [below left of=1] {};
\node[main node] (3) [below right of=1] {};

\path
(2) [->,>=stealth', shorten >=1pt] edge node  {} (1)
     edge node  {} (3);

\path     
(1) [->,>=stealth', shorten >=1pt] edge node  {} (3);     

\end{tikzpicture}

\end{tabular}

\end{center}

For another example,  the branching process can normally be shorten e.g.\ by completing the orientation of quadrilaterals as shown in Figure~\ref{completingParal}, which is {\em unique} to avoid cycles and shortcuts (the diagonal in the last case may require branching).

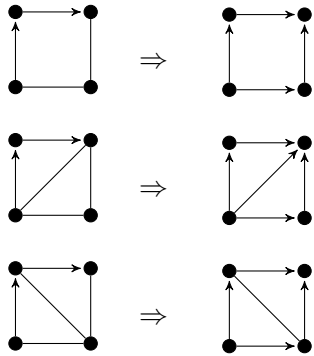
\begin{figure}
\begin{center}
\begin{tabular}{ccc}

\begin{tikzpicture}[node distance=1cm,auto,main node/.style={fill,circle,draw,inner sep=0pt,minimum size=5pt}]

\node[main node] (1) {};
\node[main node] (2) [below of=1] {};
\node[main node] (3) [right of=1] {};
\node[main node] (4) [below of=3] {};

\path
(4) edge node  {} (2)
      edge node {} (3);

\path
(2) [->,>=stealth', shorten >=1pt] edge node  {} (1);

\path     
(1) [->,>=stealth', shorten >=1pt] edge node  {} (3);     

\end{tikzpicture}

&

\begin{tikzpicture}[node distance=0.5cm,auto,main node/.style={inner sep=0pt,minimum size=5pt}]

\node [main node] (1) {\ $\Rightarrow$ \ };

\node[main node] (2) [below of=1] {};

\end{tikzpicture}

&

\begin{tikzpicture}[node distance=1cm,auto,main node/.style={fill,circle,draw,inner sep=0pt,minimum size=5pt}]

\node[main node] (1) {};
\node[main node] (2) [below of=1] {};
\node[main node] (3) [right of=1] {};
\node[main node] (4) [below of=3] {};
\node (5) [below of=2,yshift=0.9cm]{};

\path
(4) [->,>=stealth', shorten >=1pt] edge node  {} (3);    

\path
(2) [->,>=stealth', shorten >=1pt] edge node  {} (1);

\path
(2) [->,>=stealth', shorten >=1pt] edge node  {} (4);

\path     
(1) [->,>=stealth', shorten >=1pt] edge node  {} (3);     

\end{tikzpicture}

\end{tabular}

\begin{tabular}{ccc}

\begin{tikzpicture}[node distance=1cm,auto,main node/.style={fill,circle,draw,inner sep=0pt,minimum size=5pt}]

\node[main node] (1) {};
\node[main node] (2) [below of=1] {};
\node[main node] (3) [right of=1] {};
\node[main node] (4) [below of=3] {};

\path
(2) edge node  {} (3);

\path
(4) edge node  {} (2)
      edge node {} (3);

\path
(2) [->,>=stealth', shorten >=1pt] edge node  {} (1);

\path     
(1) [->,>=stealth', shorten >=1pt] edge node  {} (3);     

\end{tikzpicture}

&

\begin{tikzpicture}[node distance=0.5cm,auto,main node/.style={inner sep=0pt,minimum size=5pt}]

\node [main node] (1) {\ $\Rightarrow$ \ };

\node[main node] (2) [below of=1] {};

\end{tikzpicture}

&

\begin{tikzpicture}[node distance=1cm,auto,main node/.style={fill,circle,draw,inner sep=0pt,minimum size=5pt}]

\node[main node] (1) {};
\node[main node] (2) [below of=1] {};
\node[main node] (3) [right of=1] {};
\node[main node] (4) [below of=3] {};
\node (5) [below of=2,yshift=0.9cm]{};

\path
(2)  [->,>=stealth', shorten >=1pt]  edge node  {} (3);

\path
(4) [->,>=stealth', shorten >=1pt] edge node  {} (3);    

\path
(2) [->,>=stealth', shorten >=1pt] edge node  {} (1);

\path
(2) [->,>=stealth', shorten >=1pt] edge node  {} (4);

\path     
(1) [->,>=stealth', shorten >=1pt] edge node  {} (3);     

\end{tikzpicture}

\end{tabular}

\begin{tabular}{ccc}

\begin{tikzpicture}[node distance=1cm,auto,main node/.style={fill,circle,draw,inner sep=0pt,minimum size=5pt}]

\node[main node] (1) {};
\node[main node] (2) [below of=1] {};
\node[main node] (3) [right of=1] {};
\node[main node] (4) [below of=3] {};

\path
(4) edge node  {} (2)
      edge node {} (3)
      edge node {} (1);

\path
(2) [->,>=stealth', shorten >=1pt] edge node  {} (1);

\path     
(1) [->,>=stealth', shorten >=1pt] edge node  {} (3);     

\end{tikzpicture}

&

\begin{tikzpicture}[node distance=0.5cm,auto,main node/.style={inner sep=0pt,minimum size=5pt}]

\node [main node] (1) {\ $\Rightarrow$ \ };

\node[main node] (2) [below of=1] {};

\end{tikzpicture}

&

\begin{tikzpicture}[node distance=1cm,auto,main node/.style={fill,circle,draw,inner sep=0pt,minimum size=5pt}]

\node[main node] (1) {};
\node[main node] (2) [below of=1] {};
\node[main node] (3) [right of=1] {};
\node[main node] (4) [below of=3] {};
\node (5) [below of=2,yshift=0.9cm]{};

\path
(4) edge node  {} (1);

\path
(4) [->,>=stealth', shorten >=1pt] edge node  {} (3);    

\path
(2) [->,>=stealth', shorten >=1pt] edge node  {} (1);

\path
(2) [->,>=stealth', shorten >=1pt] edge node  {} (4);

\path     
(1) [->,>=stealth', shorten >=1pt] edge node  {} (3);     

\end{tikzpicture}

\end{tabular}
\end{center}
\vspace{-5mm}
\caption{Completing orientations of quadrilaterals}\label{completingParal}
\end{figure}

The main characterization theorem to date for word-representable graphs is the following result.

\begin{theorem}[\cite{HKP16}]\label{main-charact} A graph $G$ is word-representable if and only if $G$ admits a semi-transitive orientation.  \end{theorem}

\begin{proof}
The backwards direction is rather complicated and is omitted. An algorithm was created in \cite{HKP16} to turn a semi-transitive orientation of a graph into a word-representant.

The idea of the proof for the forward direction is as follows (see \cite{HKP16} for details). Given a word, say, $w=2421341$, orient the graph represented by $w$ by letting $x\rightarrow y$ be an edge if the {\em leftmost} $x$ is {\em to the left} of the {\em leftmost} $y$ in $w$, to obtain a semi-transitive orientation:

\begin{center}

\begin{tikzpicture}[->,>=stealth', shorten >=1pt, node distance=1cm,auto,main node/.style={circle,draw,inner sep=1pt,minimum size=5pt}]

\node[main node] (1) {1};
\node[main node] (2) [right of=1] {3};
\node[main node] (3) [below of=1] {4};
\node[main node] (4) [right of=3] {2};

\path
(1) edge node  {} (2);

\path
(3) edge node  {} (1)
      edge node  {} (2);

\path
(4) edge node  {} (3);

\end{tikzpicture}

\end{center}
\vspace{-8mm}
\end{proof}

Any {\em complete graph} is 1-representable. The algorithm in \cite{HKP16} to turn semi-transitive orientations into word-representants gave the following result.

\begin{theorem}[\cite{HKP16}]\label{main-length}
Each {\em non-complete} word-representable graph $G$ is $2(n-\kappa(G))$- representable, where $\kappa(G)$ is the size of the {\em maximum clique} in~$G$. 
\end{theorem}

As an immediate corollary of Theorem~\ref{main-length}, we have that the {\em recognition problem} of word-representability is in {\em NP}. Indeed, any word-representant is of length at most $O(n^2)$, and we need $O(n^2)$ passes through such a word to check alternation properties of all pairs of letters. There is an alternative proof of this complexity observation by Magn\'us M. Halld\'orsson in terms of {\em semi-transitive orientations}.  In presenting his proof, we follow \cite[Remark 4.2.3]{KL15}.

 Checking that a given directed graph $G$ is acyclic is a polynomially solvable problem. Indeed, it is well known that the entry $(i,j)$ of the $k$th power of the adjacency matrix of $G$ records the number of walks of length $k$ in $G$ from the vertex $i$ to the vertex $j$. Thus, if $G$ has $n$ vertices, then we need to  make sure that the diagonal entries are all $0$ in all powers, up to the $n$th power, of the adjacency matrix of $G$.  Therefore, it remains to show that it is polynomially\index{polynomially solvable problem} solvable to check that $G$ is shortcut-free.
Let $u\rightarrow v$ be an edge in $G$. Consider the induced subgraph $H_{u\rightarrow v}$ consisting of vertices ``in between'' $u$ and $v$, that is, the vertex set of $H_{u\rightarrow v}$ is 
$$\{x\ |\ \mbox{there exist directed paths from }u \mbox{ to }x\mbox{ and from } x\mbox{ to }v\}.$$ 
It is  not so difficult to prove that $u\rightarrow v$ is not a shortcut\index{shortcut} (that is, is not a shortcutting edge) if and only if $H_{u\rightarrow v}$ is transitive. Now, we can use the well known fact that finding out whether there exists a directed path from one vertex to another in a directed graph is polynomially solvable, and thus it is polynomially\index{polynomially solvable problem} solvable to determine $H_{u\rightarrow v}$ (one needs to go through $n$ vertices and check the existence of two paths for each vertex). Finally, checking transitivity is also polynomially solvable, which is not difficult to see.

The following theorem shows that word-representable graphs generalize the class of 3-colorable graphs.

\begin{theorem}[\cite{HKP16}]\label{3-col-thm} Any $3$-colorable graph is word-representable. \end{theorem}

\begin{proof}
Coloring a $3$-colorable graph in three colors, say, colors 1, 2 and 3, and orienting the edges based on the colors of their endpoints as $1 \rightarrow 2 \rightarrow 3$, we obtain a semi-transitive orientation. Indeed, obviously there are no cycles, and because the longest directed path involves only three vertices, there are no shortcuts. Theorem~\ref{main-charact} can now be applied to complete the proof.
\end{proof}

Theorem~\ref{3-col-thm} can be applied to see, for example, that the Petersen graph is word-representable, which we already know from Section~\ref{Petersen-sec}. More corollaries to Theorem~\ref{3-col-thm} can be found below.

\section{Non-word-representable Graphs}\label{non-w-r-sec}

From the discussion in Section~\ref{perm-rep-graphs-sec} we already know that the wheel graphs $W_{2n+1}$, for $n\geq 2$, are not word-representable, and that $W_5$ is the minimum (by the number of vertices) non-word-representable graph. But then, taking into account the hereditary nature of word-representability (see Remark~\ref{hereditary}), we have a family $\mathcal{W}$ of non-word-representable graphs characterised by containment of  $W_{2n+1}$ ($n\geq 2$) as an induced subgraph.

Note that {\em each} graph in $\mathcal{W}$ necessarily contains a vertex of degree 5 or more, and also a triangle as an induced subgraph. Natural questions are if there are non-word-representable graphs of {\em maximum degree} 4, and also if there are {\em triangle-free} non-word-representable graphs. Both questions were answered in affirmative. The graph to the right in Figure~\ref{2-counterex}, which was found in \cite{CKL17}, addresses the first question, while the second question is addressed by the following construction presented in \cite{HKP11}. 

Let $M$ be a 4-chromatic graph with girth at least 10 (such graphs exist by a result of Paul Erd\H{o}s; see \cite[Section 4.4]{KL15} for details). The {\em girth} of a graph is the length of a shortest cycle contained in the graph. If the graph does not contain any cycles (that is, it is an acyclic graph), its girth is defined to be infinity. Now, for every path of length 3 in $M$ add to $M$ an edge connecting path's end vertices. Then the obtained graph is triangle-free and non-word-representable \cite{HKP11}.

\subsection{Enumeration of Non-word-representable Graphs}

According to experiments run by Herman Z.Q. Chen, there are 1, 25 and 929 non-isomorphic non-word-representable connected graphs on six, seven and eight vertices, respectively.  These numbers were confirmed and extended to 68,545 for nine vertices, and 4,880,093 for 10 vertices, using a constraint programming (CP)-based method
 by  {\"{O}}zg{\"{u}}r Akg{\"{u}}n, Ian Gent and Christopher Jefferson. 
 
Figure~\ref{25-non-isom-non-word-repr} created by Chen presents the 25 non-isomorphic non-word-representable graphs on seven vertices. Note that the only non-word-representable graph on six vertices is the wheel $W_5$. Further note that the case of seven vertices gives just 10 minimal non-isomorphic non-word-representable graphs, since 15 of the graphs in Figure~\ref{25-non-isom-non-word-repr} contain $W_5$ as an induced subgraphs (these graphs are the first 11 graphs, plus the $15$th, $16$th, $18$th and $19$th graphs).

\begin{figure}[ht]
\begin{center}
\includegraphics[scale=0.8]{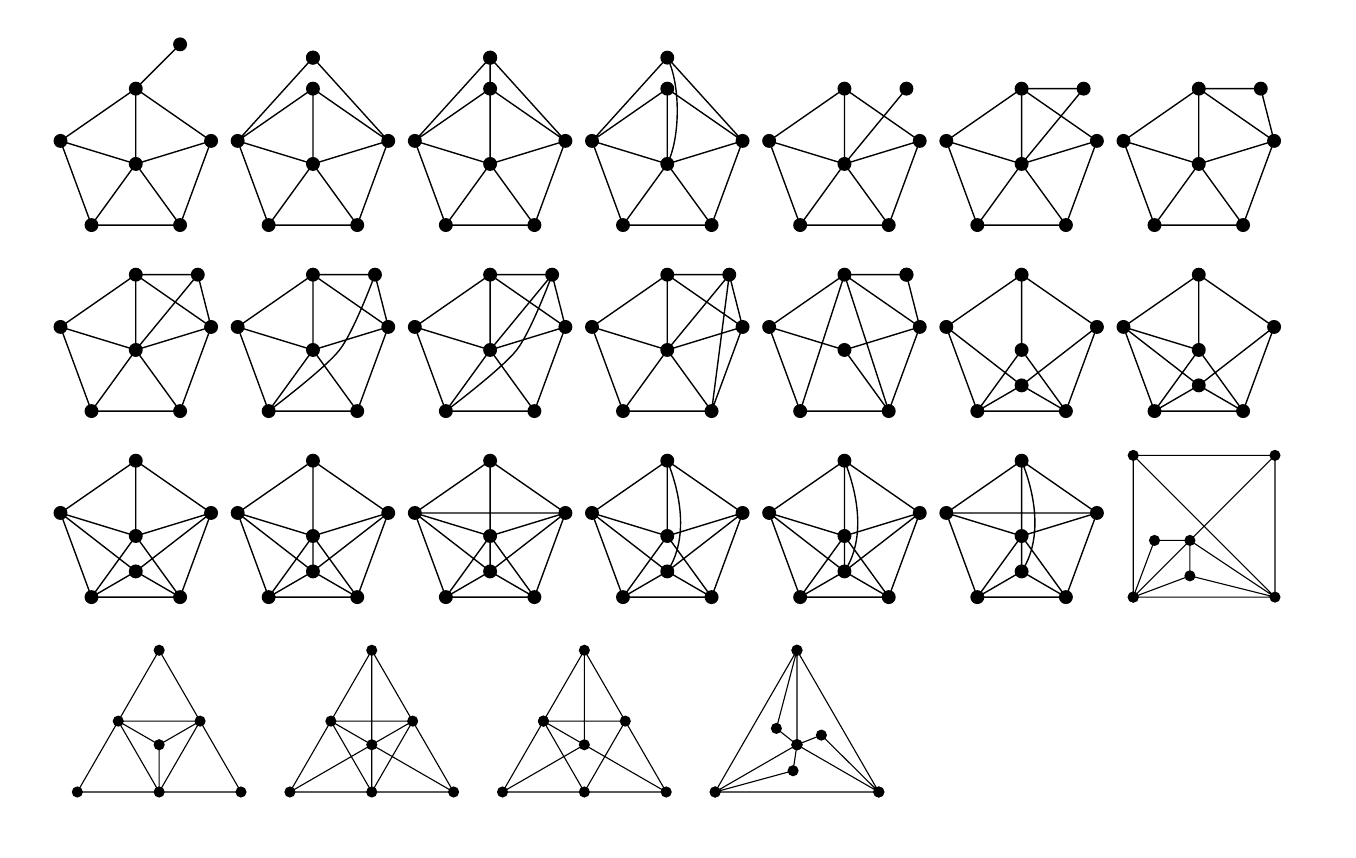}
\end{center}
\vspace{-20pt}
\caption{25  non-isomorphic non-word-representable graphs on seven vertices}\label{25-non-isom-non-word-repr}
\end{figure}

\subsection{Non-word-representable Line Graphs}\label{taking-line-sec}

The {\em line graph} of a graph $G=(V,E)$ is the graph with vertex set $E$ in which two vertices are 
adjacent if and only if the corresponding edges of $G$ share a vertex. The line graph of $G$ is 
denoted $L(G)$. Line graphs give a tool to construct non-word-representable graphs as follows from the theorems below. 

\begin{theorem}[\cite{KSSU11}] Let $n\geq 4$. For {\em any} wheel graph $W_n$, the line graph $L(W_n)$ is non-word-representable. \end{theorem}

\begin{theorem}[\cite{KSSU11}] Let $n\geq 5$. For {\em any} complete graph $K_n$, the line graph $L(K_n)$ is non-word-representable. \end{theorem}

\begin{figure}

\begin{center}
\begin{tabular}{ccccc}

\begin{tikzpicture}[node distance=0.8cm,auto,main node/.style={fill,circle,draw,inner sep=0pt,minimum size=5pt}]

\node[main node] (1) {};
\node[main node] (2) [above right of=1] {};
\node[main node] (3) [right of=1] {};
\node[main node] (4) [below right of=1] {};
\node (5) [left of=1] {$K_{1,3}=$};

\path
(1) edge node  {} (2)
     edge node  {} (3)
     edge node  {} (4);

\end{tikzpicture}

&

&

\begin{tikzpicture}[node distance=0.8cm,auto,main node/.style={fill,circle,draw,inner sep=0pt,minimum size=5pt}]

\node[main node] (1) {};
\node[main node] (2) [right of=1] {};
\node[main node] (3) [below of=2] {};
\node[main node] (4) [left of=3] {};
\node (5) [below left of=1] {$C_{4}=$};

\path
(1) edge node  {} (2)
     edge node  {} (4);

\path
(3) edge node  {} (2)
     edge node  {} (4);

\end{tikzpicture}

&

&

\begin{tikzpicture}[node distance=0.8cm,auto,main node/.style={fill,circle,draw,inner sep=0pt,minimum size=5pt}]

\node[main node] (1) {};
\node[main node] (2) [right of=1] {};
\node[main node] (3) [right of=2] {};
\node[main node] (4) [right of=3] {};
\node (5) [left of=1] {$P_{4}=$};

\path
(2) edge node  {} (1)
     edge node  {} (3);

\path
(3) edge node  {} (4);

\end{tikzpicture}

\end{tabular}

\end{center}
\vspace{-5mm}
\caption{The claw graph $K_{1,3}$, the cycle graph $C_4$, and the path graph $P_4$}\label{just-three-graphs}
\end{figure}
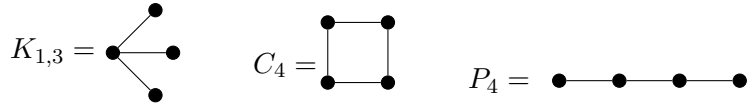

The following theorem is especially interesting as it shows how to turn essentially any graph into non-word-representable graph.

\begin{theorem}[\cite{KSSU11}] If a connected graph $G$ is {\em not} a path graph, a cycle graph, or the claw graph $K_{1,3}$, then the line graph $L^n(G)$ obtained by application of the procedure of taking the line graph to $G$ $n$ times, is non-word-representable for $n\geq 4$. \end{theorem}

\section{Word-Representability and Operations on Graphs}

In this section we consider some of the most basic operations on graphs, namely,  taking the complement, edge subdivision, edge contraction, connecting two graphs by an edge and gluing two graphs in a clique, replacing a vertex with a module, Cartesian product, rooted product and taking line graph.

We do not consider edge-addition/deletion trivially not preserving (non)-word-representability, although there are situations when these operations may preserve word-representability. For example, it is shown in \cite{KK17} that edge-deletion preserves word-representability on $K_4$-free word-representable graphs.

Finally, we do not discuss the operations $Y\rightarrow\Delta$ (replacing an induced subgraph $K_{1,3}$, the claw, on vertices $v_0,v_1,v_2,v_3$, where $v_0$ is the apex, by the triangle on vertices $v_1,v_2,v_3$ and removing $v_0$) and $\Delta\rightarrow Y$ (removing the edges of a triangle on vertices $v_1,v_2,v_3$ and adding a vertex $v_0$ connected to $v_1,v_2,v_3$) recently studied in \cite{KK17} in the context of word-representability of graphs.

\subsection{Taking the Complement}\label{taking-complem-subseq}

Starting with a word-representable graph and taking its complement, we may either obtain a word-representable graph or not. Indeed, for example, both any graph on at most five vertices and its complement are word-representable. On the other hand, let $G$ be the graph formed by the $5$-cycle (2,4,6,3,5) and an isolated vertex~1. The $5$-cycle can be represented by the word $2542643653$ (see Section~\ref{cycle-sec} for a technique to represent cycle graphs) and thus the graph $G$ can be represented by the word $112542643653$. However, taking the complement of $G$, we obtain the wheel graph $W_5$, which is not word-representable.  

Similarly, starting with a non-word-representable graph and taking its complement, we can either obtain a word-representable graph or not. Indeed, the complement of the non-word-representable wheel $W_5$ is word-representable, as is discussed above. On the other hand, the graph $G$ having two connected components, one $W_5$ and the other one the $5$-cycle $C_5$, is non-word-representable because of the induced subgraph $W_5$, while the complement of $G$ also contains an induced subgraph $W_5$ (formed by the vertices of $C_5$ in $G$ and any of the remaining vertices) and thus is also non-word-representable. 

\subsection{Edge Subdivision and Edge Contraction}

Subdivision of graphs (see Definition~\ref{def-subdivision}) is based on subdivision of individual edges, and it is considered in Section~\ref{subdiv-graphs-3} from 3-representability point of view. 

If we change ``3-representable'' by ``word-representable'' in Theorem~\ref{minor-3-word-repr} we would obtain a weaker, but clearly still true statement, which is not hard to prove directly via  semi-transitive orientations. Indeed, each path of length at least 3 added instead of an edge $e$ can be oriented in a ``blocking'' way, so that there would be no directed path between $e$'s endpoints. Thus, edge subdivision does not preserve the property of being non-word-representable.  The following theorem shows that edge subdivision may be preserved on some subclasses of word-representable graphs, but not on the others.

\begin{theorem}[\cite{KK17}] Edge subdivision preserves word-representability on $K_4$-free word-representable graphs, and it does not necessarily preserve word-representability on $K_5$-free word-representable graphs.\end{theorem} 

Recall the definition of edge contraction in Definition~\ref{edge-contraction-undirected}. By Theorem~\ref{minor-3-word-repr}, contracting an edge in a word-representable graph may result in a non-word-representable graph, while in many cases, e.g. in the case of path graphs, word-representability is preserved under this operation. 

On the other hand, when starting from a non-word-representable graph, a graph obtained from it by edge contraction can also be either word-representable or non-word-representable. For example, contracting any edge incident with the bottommost vertex in the non-word-representable  graph to the right in Figure~\ref{2-counterex}, we obtain a graph on six vertices that is different from $W_5$ and is thus word-representable. Finally, any non-word-representable graph can be glued in a vertex with a path graph $P$ (the resulting graph will be non-word-representable), so that contracting any edge in the subgraph formed by $P$ results in a non-word-representable graph.

\subsection{Connecting two Graphs by an Edge and Gluing two Graphs in a Clique}\label{glueing-connecting-two-graphs}

In what follows, by {\em gluing} two graphs in a clique we mean the following operation. Suppose $a_1,\ldots,a_k$ and $b_1,\ldots,b_k$ are cliques of size $k$ in graphs $G_1$ and $G_2$, respectively. Then gluing $G_1$ and $G_2$ in a clique of size $k$ means identifying each $a_i$ with one $b_j$, for $i,j\in\{1,\ldots,k\}$ so that the neighbourhood of the obtained vertex $c_{i,j}$ is the union of the neighbourhoods of $a_i$ and $b_j$. 

By the hereditary nature of word-representability (see Remark~\ref{hereditary}), if at least one of two graphs is non-word-representable, then gluing the graphs in a clique, or connecting two graphs by an edge (with the endpoints belonging to different graphs) will result in a non-word-representable graph. 

On the other hand, suppose that graphs $G_1$ and $G_2$ are word-representable. Then gluing the graphs in a vertex, or connecting the graphs by an edge will result in a word-representable graph. The latter statement is easy to see using the notion of semi-transitive orientation. Indeed, by Theorem~\ref{main-charact} both $G_1$ and $G_2$ can be oriented semi-transitively, and  gluing the oriented graphs in a vertex, or connecting the graphs by an edge oriented arbitrarily, will not result in any cycles or shortcuts created. In fact, it was shown in~\cite{K13} that if $G_1$ is $k_1$-representable (such a $k_1$ must exist by Theorem~\ref{equiv-thm}) and $G_2$ is $k_2$-representable, then essentially always the graph obtained either by gluing $G_1$ and $G_2$ in a vertex or by connecting the graphs by an edge is $\max(k_1,k_2)$-representable.  

Even though glueing two word-representable graphs in a vertex (clique of size 1) always results in a word-representable graph, this is not necessarily true for glueing graphs in an edge (clique of size 2) or in a triangle (clique of size 3). We refer to \cite[Section 5.4.3]{KL15} for the respective examples. Glueing two graphs in cliques of size 4 or more in the context of word-representability remains an unexplored direction.

\subsection{Replacing a Vertex with a Module}\label{modules-subsec}

A subset $X$ of the set of vertices $V$ of a graph $G$ is a {\em module} if all members of $X$ have the same set of neighbours among vertices not in $X$ (that is, among vertices in $V\setminus X$). For example, Figure~\ref{modules-business} shows replacement of the vertex $1$ in the triangular prism by the module $K_3$ formed by the vertices $a$, $b$ and $c$. Thus, $\{a,b,c\}$ is a module of the graph on the right in Figure~\ref{modules-business}.

\begin{figure}[ht]
\begin{center}
\includegraphics[scale=0.6]{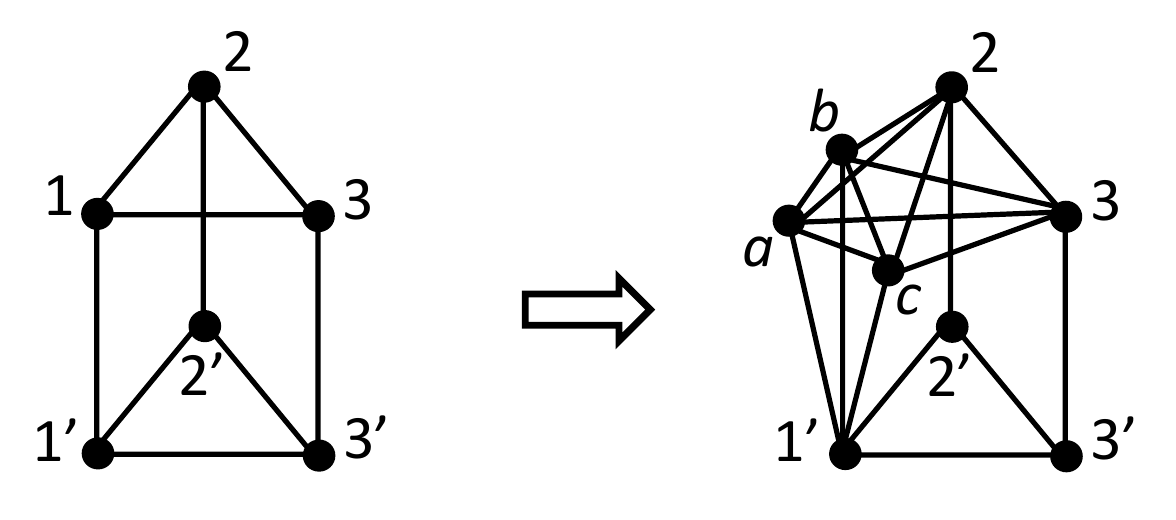}
\end{center}
\vspace{-20pt}
\caption{Replacing a vertex by a module}
\label{modules-business}
\end{figure}

\begin{theorem}[\cite{K13}]\label{module-thm} Suppose that $G=(V,E)$ is a word-representable graph and $x\in V$. Let $G'$ be obtained from $G$ by replacing $x$ with a module $M$, where $M$ is any comparability graph (in particular, any clique). Then $G'$ is also word-representable. Moreover, if $\mathcal{R}(G)=k_1$ and $\mathcal{R}(M)=k_2$ then $\mathcal{R}(G')=k$, where $k=\max\{k_1,k_2\}$. \end{theorem}

\subsection{Cartesian Product of two Graphs}\label{cartesian-product}

The {\em Cartesian product} $G \square H$ of graphs $G=(V(G),E(G))$ and $H=(V(H),E(H))$ is a graph such that
\begin{itemize}
\item the vertex set of $G \square H$ is the Cartesian product $V(G) \times V(H)$; and
\item any two vertices $(u,u')$ and $(v,v')$ are adjacent in $G \square H$ if and only if either
\begin{itemize}
\item $u = v$ and $u'$ is adjacent to $v'$ in $H$, or
\item $u' = v'$ and $u$ is adjacent to $v$ in $G$. 
\end{itemize}
\end{itemize}

See Figure~\ref{cartesian-prod-example} for an example of the Cartesian product of two graphs.

\begin{figure}[h]
\begin{center}

\setlength{\unitlength}{5mm}

\begin{picture}(14,4)

\put(0,0){

\put(1,0){\p}\put(0,2){\p}\put(2,1){\p}\put(1,3){\p}\put(1,4){\p}

\put(1,0){\line(1,1){1}}
\put(0,2){\line(1,1){1}}
\put(1,3){\line(0,1){1}}
\put(0,2){\line(1,-2){1}}
\put(1,3){\line(1,-2){1}}

\put(3.1,1.6){$\square$}

\put(5,1.8){\line(1,0){1}}
\put(5,1.8){\p}\put(6,1.8){\p}

\put(7.1,1.6){$=$}

\put(10,0){\p}\put(9,2){\p}\put(11,1){\p}\put(10,3){\p}\put(10,4){\p}

\put(10,0){\line(1,1){1}}
\put(9,2){\line(1,1){1}}
\put(10,3){\line(0,1){1}}
\put(9,2){\line(1,-2){1}}
\put(10,3){\line(1,-2){1}}

\put(13,0){\p}\put(12,2){\p}\put(14,1){\p}\put(13,3){\p}\put(13,4){\p}

\put(13,0){\line(1,1){1}}
\put(12,2){\line(1,1){1}}
\put(13,3){\line(0,1){1}}
\put(12,2){\line(1,-2){1}}
\put(13,3){\line(1,-2){1}}

\put(10,0){\line(1,0){3}}\put(9,2){\line(1,0){3}}\put(11,1){\line(1,0){3}}\put(10,3){\line(1,0){3}}\put(10,4){\line(1,0){3}}

}

\end{picture}
\caption{Cartesian product of two graphs}\label{cartesian-prod-example}
\end{center}
\end{figure}
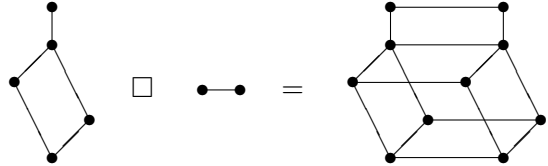

A proof of the following  theorem was given by Bruce Sagan in 2014. The proof relies on semi-transitive orientations and it can be found in \cite[Section 5.4.5]{KL15}.

\begin{theorem}[Sagan]\label{cartesian-product-thm} Let $G$ and $H$ be two word-representable graphs. Then the Cartesian product $G\square H$ is also word-representable. \end{theorem}

\subsection{Rooted Product of Graphs}\label{rooted-product-graphs-sec}

The {\em rooted product} of a graph $G$ and a rooted graph $H$ (i.e. one vertex of $H$ is distinguished), $G\circ H$, is defined as follows: take $|V(G)|$ copies of $H$, and for every vertex $v_i$ of $G$, identify $v_i$ with the root vertex of the $i$th copy of $H$. See Figure~\ref{rooted-prod-example} for an example of the rooted product of two graphs.

\begin{figure}[h]
\begin{center}

\setlength{\unitlength}{5mm}

\begin{picture}(14,4)

\put(0,0){

\put(1,0){\p}\put(0,2){\p}\put(2,1){\p}\put(1,3){\p}\put(1,4){\p}

\put(1,0){\line(1,1){1}}
\put(0,2){\line(1,1){1}}
\put(1,3){\line(0,1){1}}
\put(0,2){\line(1,-2){1}}
\put(1,3){\line(1,-2){1}}

\put(3.1,1.6){$\times$}

\put(5,1.8){\line(1,0){1}}
\put(5,1.8){\p}\put(6,1.8){\p}

\put(7.1,1.6){$=$}

\put(10,0){\p}\put(9,2){\p}\put(11,1){\p}\put(10,3){\p}\put(10,4){\p}

\put(10,0){\line(1,1){1}}
\put(9,2){\line(1,1){1}}
\put(10,3){\line(0,1){1}}
\put(9,2){\line(1,-2){1}}
\put(10,3){\line(1,-2){1}}

\put(13,0){\p}\put(12,2){\p}\put(14,1){\p}\put(13,3){\p}\put(13,4){\p}

\put(10,0){\line(1,0){3}}\put(9,2){\line(1,0){3}}\put(11,1){\line(1,0){3}}\put(10,3){\line(1,0){3}}\put(10,4){\line(1,0){3}}

}

\end{picture}
\caption{Rooted product of two graphs}\label{rooted-prod-example}
\end{center}
\end{figure}
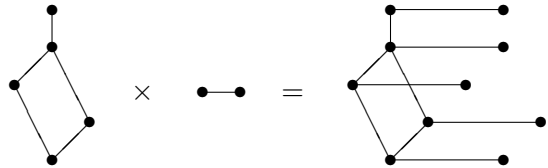

The next theorem is an analogue of Theorem~\ref{cartesian-product-thm} for the rooted product of two graphs.  
 
\begin{theorem}[\cite{KL15}]\label{rooted-product-thm} Let $G$ and $H$ be two word-representable graphs. Then the rooted product $G\circ H$ is also word-representable. \end{theorem}

\begin{proof} Identifying a vertex $v_i$ in $G$ with the root vertex of the $i$th copy of $H$ in the definition of the rooted product gives a word-representable graph by the discussion in Section~\ref{glueing-connecting-two-graphs}. Thus, identifying the root vertices, one by one, we will keep obtaining word-representable graphs, which gives us at the end word-representability of $G\circ H$.\end{proof}

\subsection{Taking the Line Graph Operation}

Taking the line graph operation has already been considered in Section~\ref{taking-line-sec}. Based on the results presented in that section, we can see that this operation can turn a word-representable graph into either a word-representable graph or non-word-representable graph. Also, there are examples of when the line graph of a non-word-representable graph is non-word-representable. However, it remains an open problem whether a non-word-representable graph can be turned into a word-representable graph by applying the line graph operation.

\section{Computational Complexity Results and Word-Representability of Planar Graphs}

In this section we will present known complexity results and also discuss word-representability of planar graphs.

\subsection{A Summary of Known Complexity Results}

Even though the Maximum Clique problem is polynomially solvable on word-representable graphs (see Theorem~\ref{Max-Click-thm}), many classical optimization problems are NP-hard on these graphs. The latter follows from the problems being NP-hard on 3-colorable graphs and Theorem~\ref{3-col-thm}. 

The justification of the known complexity results presented in Table~\ref{known-complexities}, as well as the definitions of the problems can be found in \cite[Section 4.2]{KL15}. However, below we discuss a proof of the fact that recognizing word-representability is an NP-complete problem. We refer to \cite[Section 4.2]{KL15} for any missed references to the results we use.

\begin{table}
\begin{center}
\begin{tabular}{c|c}
{\bf problem} & {\bf complexity}  \\
\hline
deciding whether a given graph is word-representable & NP-complete\index{NP-complete problem} \\
\hline
approximating the graph representation number  & NP-hard\index{NP-hard problem} \\
within a factor of
$n^{1-\epsilon}$ for any $\epsilon > 0$ & \\
\hline
Clique\index{Clique Covering problem} Covering & NP-hard \\
\hline
deciding whether a given graph is
$k$-word-representable  & NP-complete \\
for any fixed $k$, $3\leq k\leq\lceil n/2 \rceil$ & \\
\hline
Dominating\index{Dominating Set problem} Set &  NP-hard\\
\hline
Vertex Colouring\index{Vertex Colouring problem} &  NP-hard\\
\hline
Maximum\index{Maximum Clique problem} Clique & in P\\
\hline
Maximum\index{Maximum Independent Set problem} Independent Set & NP-hard\\
\hline
\end{tabular}
\caption{Known complexities for problems on word-representable\index{word-representable graph} graphs}\label{known-complexities}
\end{center}
\end{table} 

Suppose that $P$ is a poset and $x$ and $y$ are two of its elements. We say that $x$ {\em covers} $y$ if $x>y$ and there is no element $z$ in $P$ such that $x>z>y$.

The {\em cover graph} $G_P$ of a poset $P$ has $P$'s elements as its vertices, and $\{x,y\}$ is an edge in $G_P$ if and only if either $x$ covers $y$, or vice versa. The {\em diagram} of $P$, sometimes called a {\em Hasse diagram} or {\em order diagram}, is a drawing of the cover graph of $G$ in the plane with $x$ being higher than $y$ whenever $x$ covers $y$ in $P$. The three-dimensional cube in Figure~\ref{prisms-3-4} is an example of a cover graph. 

Vincent Limouzy observed in 2014 that semi-transitive orientations of triangle-free graphs are exactly the {\em $2$-good orientations} considered in~\cite{Pretzel} by Pretzel (we refer to that paper for the definition of a $k$-good orientation). Thus, by Proposition~1 in  \cite{Pretzel} we have the following reformulation of Pretzel's result in our language.

\begin{theorem}[Limouzy] The class of triangle-free word-representable graphs is exactly the class of cover graphs of posets. \end{theorem}

It was further observed by  Limouzy, that it is an NP-complete problem to recognize the class of cover graphs of posets. This implies the following theorem, which is a key complexity result on word-representable graphs.

\begin{theorem}[Limouzy]\label{recogn-word-repr-NP-complete} It is an NP-complete problem to recognize whether a given graph is word-representable. \end{theorem}

\subsection{Word-Representability of Planar Graphs}

Recall that not all planar graphs are word-representable. Indeed, for example, wheel graphs $W_{2n+1}$, or graphs in Figure~\ref{2-counterex}, are not word-representable. 

\begin{theorem}[\cite{HKP11}]  Triangle-free planar graphs are word-representable. \end{theorem}

\begin{proof} By Gr\"otzch's theorem \cite{Thomassen}, {\em every} triangle-free planar graph is $3$-colorable, and Theorem~\ref{3-col-thm} can be applied.\end{proof}

It remains a challenging open problem to classify word-representable planar graphs. Towards solving the problem, various {\em triangulations} and certain {\em subdivisions} of planar graphs were considered to be discussed next. Key tools to study word-representability of planar graphs are {\em $3$-colorability} and {\em semi-transitive orientations}.

\subsubsection{Word-Representability of Polyomino Triangulations}

A {\em polyomino} is a plane geometric figure formed by joining one or more equal squares edge to edge. 
Letting corners of squares in a polyomino be vertices, we can treat polyominoes as graphs. In particular, well known {\em grid graphs} are obtained from polyominoes in this way. Of particular interest to us are  {\em convex polyominoes}. A polyomino is said to be {\em column convex} if its intersection with any vertical line is convex (in other words, each column has no holes). Similarly, a polyomino is said to be {\em row convex} if its intersection with any horizontal line is convex. A polyomino is said to be {\em convex} if it is row and column convex.

When dealing with word-representability of triangulations of convex polyominoes (such as in Figure~\ref{polyom-triangl}), one should watch for {\em odd wheel graphs} as induced subgraphs (such as the part of the graph in bold in Figure~\ref{polyom-triangl}). Absence of such subgraphs will imply 3-colorability and thus word-representability, which is the basis of the proof of the following theorem. 

\begin{theorem}[\cite{AKM15}]\label{tri-con-pol} A triangulation of a {\em convex} polyomino is word-representable if and only if it is $3$-colorable. There are non-$3$-colorable word-representable {\em non-convex} polyomino triangulations. \end{theorem}

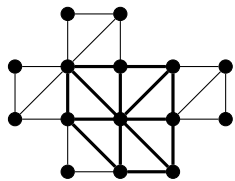
\begin{figure} 
\begin{center}
\begin{tikzpicture}[node distance=0.7cm,auto,main node/.style={fill,circle,draw,inner sep=0pt,minimum size=5pt}]

\node[main node] (1) {};
\node[main node] (2) [right of=1] {};
\node[main node] (3) [below of=1] {};
\node[main node] (4) [right of=3] {};
\node[main node] (5) [left of=3] {};
\node[main node] (6) [below of=5] {};
\node[main node] (7) [right of=6] {};
\node[main node] (8) [right of=7] {};
\node[main node] (9) [right of=4] {};
\node[main node] (10) [right of=9] {};
\node[main node] (11) [right of=8] {};
\node[main node] (12) [right of=11] {};
\node[main node] (13) [below of=7] {};
\node[main node] (14) [right of=13] {};
\node[main node] (15) [right of=14] {};

\path
(1) edge node  {} (13)
     edge node  {} (2);

\path
(2) edge node  {} (14)
     edge node  {} (3);

\path
(5) edge node  {} (10)
     edge node  {} (6);

\path
(6) edge node  {} (3)
     edge node  {} (12);

\path
(15) edge node  {} (13)
       edge node  {} (9)
       edge node  {} (3);

\path
(7) edge node  {} (14);

\path
(8) edge node  {} (9);

\path
(10) edge node  {} (11)
     edge node  {} (12);

\path
(3) [very thick] edge node  {} (9)
      [very thick] edge node  {} (15)
     [very thick] edge node  {} (7);

\path
(14) [very thick] edge node  {} (4)
      [very thick] edge node  {} (15)
     [very thick] edge node  {} (7);

\path
(9) [very thick] edge node  {} (8)
      [very thick] edge node  {} (15);

\path
(7) [very thick] edge node  {} (11);

\end{tikzpicture}

\vspace{-5mm}
\caption{A triangulation of a poyomino}\label{polyom-triangl}
\end{center}

\end{figure}

The case of rectangular polyomino triangulations with a single domino tile (such as in Figure~\ref{rec-polyom-triangl}) is considered in the next theorem.

\begin{theorem}[\cite{GK18}]\label{thm-rec-po-tri} A triangulation of a rectangular polyomino with a single domino tile is 
word-representable if and only if it is $3$-colorable. \end{theorem}

\begin{figure} 
\begin{center}
\begin{tikzpicture}[node distance=0.7cm,auto,main node/.style={fill,circle,draw,inner sep=0pt,minimum size=5pt}]

\node[main node] (1) {};
\node[main node] (2) [right of=1] {};
\node[main node] (3) [below of=1] {};
\node[main node] (4) [right of=3] {};
\node[main node] (5) [left of=3] {};
\node[main node] (6) [below of=5] {};
\node[main node] (7) [right of=6] {};
\node[main node] (8) [right of=7] {};
\node[main node] (9) [right of=4] {};
\node[main node] (10) [right of=9] {};
\node[main node] (11) [right of=8] {};
\node[main node] (12) [right of=11] {};
\node[main node] (13) [below of=7] {};
\node[main node] (14) [right of=13] {};
\node[main node] (15) [right of=14] {};
\node[main node] (16) [left of=13] {};
\node[main node] (17) [right of=15] {};
\node[main node] (18) [left of=1] {};
\node[main node] (19) [right of=2] {};
\node[main node] (20) [right of=19] {};

\path
(1) edge node  {} (13)
     edge node  {} (2);

\path
(20) edge node  {} (2)
     edge node  {} (10);

\path
(18) edge node  {} (1)
      edge node  {} (3)
     edge node  {} (5);

\path
(16) edge node  {} (6)
     edge node  {} (7)
     edge node  {} (13);

\path
(17) edge node  {} (11)
     edge node  {} (12)
     edge node  {} (15);

\path
(2) edge node  {} (14)
     edge node  {} (3);

\path
(5) edge node  {} (10)
     edge node  {} (6);

\path
(6) edge node  {} (3)
     edge node  {} (12);

\path
(15) edge node  {} (13)
       edge node  {} (11)
       edge node  {} (3);

\path
(7) edge node  {} (14);

\path
(11) edge node  {} (10);

\path
(8) edge node  {} (20)
     edge node  {} (10);

\path
(10) edge node  {} (11)
     edge node  {} (12);

\path
(19) edge node  {} (4)
     edge node  {} (9);

\path
(4) [very thick] edge node  {} (10)
       [very thick] edge node  {} (8);

\path
(12) [very thick] edge node  {} (10)
       [very thick] edge node  {} (8);

\end{tikzpicture}

\vspace{-5mm}
\caption{A triangulation of a rectangular polyomino with a single domino tile}\label{rec-polyom-triangl}
\end{center}

\end{figure}
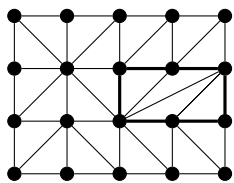

\subsubsection{Word-Representability of Near-Triangulations}

A {\em near-triangulation} is a planar graph in which each {\em inner bounded face} is a triangle (where the {\em outer face} may possibly not be a triangle).

The following theorem is a far-reaching generalization of Theorems~\ref{tri-con-pol} and~\ref{thm-rec-po-tri}.

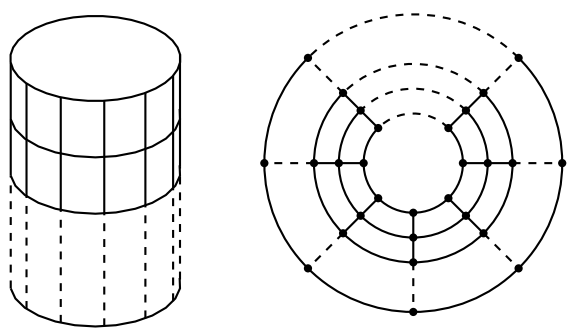
\begin{figure}
 \begin{center}
\begin{tikzpicture}[scale=1,x={(-0.8cm,-0.4cm)},y={(0.8cm,-0.4cm)},
    z={(0cm,1cm)},font=\large]
	\def\h{1.5}

	\foreach \t in {-39,-9,...,154}
		\draw[thick] ({cos(\t)},{sin(\t)},{1*\h}) --({cos(\t)},{sin(\t)},{2.0*\h});
	\foreach \t in {-39,-9,...,154}
		\draw[thick,dashed] ({cos(\t)},{sin(\t)},0) --({cos(\t)},{sin(\t)},{1*\h});

	\draw[,thick] ({cos(-39)},{sin(-39)},0) 
		\foreach \t in {-39,-25,...,154}
			{--({cos(\t)},{sin(\t)},0)};
			
	\draw[,thick] (1,0,{2*\h}) 
		\foreach \t in {10,20,...,360}
			{--({cos(\t)},{sin(\t)},{2*\h})}--cycle;
			
	\draw[,thick] ({cos(-39)},{sin(-39)},{1.5*\h}) 
		\foreach \t in {-39,-25,...,154}
			{--({cos(\t)},{sin(\t)},{1.5*\h})};
				
	\draw[,thick] ({cos(-39)},{sin(-39)},{1*\h}) 
		\foreach \t in {-39,-25,...,154}
			{--({cos(\t)},{sin(\t)},{1*\h})};
		
\end{tikzpicture}
\qquad
\begin{tikzpicture} [scale=.66, line join=round, >=latex,thick]
\foreach \r in {1,1.5,2,3}
{
\draw[dashed] ([shift=(45:\r cm)]0,0) arc (45:135:\r cm);
\draw[] ([shift=(45:\r cm)]0,0) arc (45:-225:\r cm);
}
\foreach \a in {-2,-1,0,1,-3,-4,-5}
{
\draw (45*\a:1 cm)-- (45*\a:2 cm);
\draw[dashed] (45*\a:2 cm)-- (45*\a:3 cm);
 \fill[black!100] (45*\a:1 cm) circle(0.5ex)
                 (45*\a:1.5 cm) circle(0.5ex)
                 (45*\a:2 cm) circle(0.5ex)
                 (45*\a:3 cm) circle(0.5ex)
                ;
}
\end{tikzpicture}
\caption{Grid-covered cylinder}\label{grid-covered-cylind-pic}
\end{center}
\end{figure}

\begin{theorem}[\cite{G16}] A $K_4$-free near-triangulation is $3$-colorable if and only if it is word-representable. \end{theorem}

Characterization of word-representable near-triangulations (containing $K_4$) is still an open problem.

\subsubsection{Triangulations of Grid-covered Cylinder Graphs}

A {\em grid-covered cylinder}, {\em GCC} for brevity, is a 3-dimensional figure formed by drawing vertical lines and horizontal circles on the surface of a cylinder, each of which are parallel to the generating line and the upper face of the cylinder, respectively. A GCC can be thought of as the object obtained by gluing the left and right sides of a rectangular grid. See the left picture in Figure~\ref{grid-covered-cylind-pic} for a schematic way to draw a GCC. The vertical lines and horizontal circles are called the {\em grid lines}. The part of a GCC between two consecutive vertical lines defines a {\em sector}.

Any GCC defines a graph, called {\em grid-covered cylinder graph}, or {\em GCCG}, whose set of vertices is given by intersection of the grid lines, and whose edges are parts of grid lines between the respective vertices.  A typical triangulation of a GCCG is presented schematically in Figure~\ref{triang-grid-cover-cylinder-general}.

Word-representability of triangulations of any GCCG is completely characterized by the following two theorems, which take into consideration the number of sectors in a GCCG.

\begin{theorem}[\cite{CKS16A}]\label{thm-morethan-3} A triangulation of a GCCG with {\em more than three} sectors is word-representable if and only if it contains no $W_5$ or $W_7$ as an induced subgraph. \end{theorem}

 \begin{figure}
 \begin{center}
\begin{tikzpicture} [scale=0.7, line join=round, >=latex, thick]
\foreach \r in {1,1.5,2,3}
{
\draw[dashed] ([shift=(45:\r cm)]0,0) arc (45:135:\r cm);
\draw[] ([shift=(45:\r cm)]0,0) arc (45:-225:\r cm);
}
\foreach \a in {-2,-1,0,1,-3,-4,-5}
{
\draw (45*\a:1 cm)-- (45*\a:2 cm);
\draw[dashed] (45*\a:2 cm)-- (45*\a:3 cm);
 \fill[black!100] (45*\a:1 cm) circle(0.5ex)
                 (45*\a:1.5 cm) circle(0.5ex)
                 (45*\a:2 cm) circle(0.5ex)
                 (45*\a:3 cm) circle(0.5ex)
                ;
}

\draw (45:1 cm)--(0:1.5 cm)--(45:2 cm)
      (0:1.5 cm)--(-45:1 cm)
      (0:2 cm)--(-45:1.5 cm)
      ;
\draw (-45:1 cm)--(-90:1.5 cm)--(-45:2 cm)
      (-90:1.5 cm)--(-135:1 cm)
      (-90:2 cm)--(-135:1.5 cm)
      --(-180:1 cm)--(-225:1.5 cm)-- (-180:2 cm)--(-135:1.5 cm)
      ;
\end{tikzpicture}
\caption{A triangulation of a GCCG}\label{triang-grid-cover-cylinder-general}
\end{center}
\end{figure}
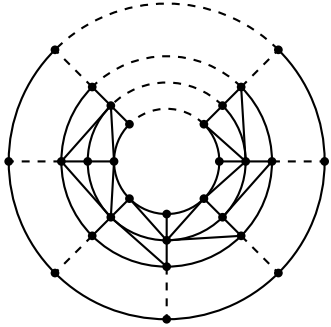

\begin{theorem}[\cite{CKS16A}]\label{thm-3} A triangulation of a GCCG with three sectors is word-representable if and only if it contains no graph in Figure~{\em \ref{non-repr-induced-subgraphs}} as an induced subgraph.\end{theorem}

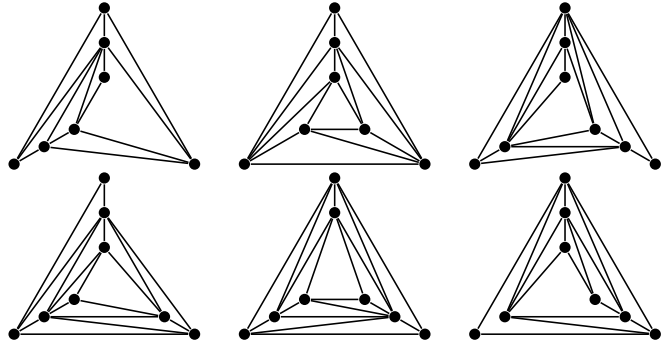
\begin{figure}[!htbp]
\begin{center}
\begin{tikzpicture}[line width=.6pt,scale=0.8,dot/.style={circle,inner sep=1.5pt,fill,name=#1}]
\node[dot=L10] at (-0.5, -0.289) {} ;
\node[dot=L12] at (0, 0.577) {} ;
\node[dot=L20] at (-1, -0.577) {} ;
\node[dot=L22] at (0, 1.154) {} ;
\node[dot=L30] at (-1.5, -0.867) {} ;
\node[dot=L21] at (1.5, -0.867) {} ;
\node[dot=L32] at (0, 1.731) {} ;
\draw[] (L12)-- (L10) ;
\draw[] (L10)-- (L20) ;
\draw[] (L20)-- (L21) ;
\draw[] (L10)-- (L21) ;
\draw[] (L21)-- (L22) ;
\draw[] (L22)-- (L20) ;
\draw[] (L12)-- (L22) ;
\draw[] (L10)-- (L22) ;
\draw[] (L20)-- (L30) ;
\draw[] (L22)-- (L30) ;
\draw[] (L32)-- (L30) ;
\draw[] (L22)-- (L32) ;
\draw[] (L21)-- (L32) ;
\end{tikzpicture}
\quad
\begin{tikzpicture}[line width=.6pt,scale=0.8,dot/.style={circle,inner sep=1.5pt,fill,name=#1}]
\node[dot=L10] at (-0.5, -0.289) {} ;
\node[dot=L11] at (0.5, -0.289) {} ;
\node[dot=L12] at (0, 0.577) {} ;
\node[dot=L22] at (0, 1.154) {} ;
\node[dot=L20] at (-1.5, -0.867) {} ;
\node[dot=L21] at (1.5, -0.867) {} ;
\node[dot=L32] at (0, 1.731) {} ;
\draw[] (L10)-- (L11) ;
\draw[] (L11)-- (L12) ;
\draw[] (L12)-- (L10) ;
\draw[] (L10)-- (L20) ;
\draw[] (L12)-- (L20) ;
\draw[] (L20)-- (L21) ;
\draw[] (L11)-- (L21) ;
\draw[] (L10)-- (L21) ;
\draw[] (L21)-- (L22) ;
\draw[] (L22)-- (L20) ;
\draw[] (L12)-- (L22) ;
\draw[] (L11)-- (L22) ;
\draw[] (L22)-- (L32) ;
\draw[] (L21)-- (L32) ;
\draw[] (L20)-- (L32) ;
\end{tikzpicture}
\quad
\begin{tikzpicture}[line width=.6pt,scale=0.8,dot/.style={circle,inner sep=1.5pt,fill,name=#1}]
\node[dot=L21] at (0.5, -0.289) {} ;
\node[dot=L12] at (0, 0.577) {} ;
\node[dot=L20] at (-1, -0.577) {} ;
\node[dot=L31] at (1, -0.577) {} ;
\node[dot=L22] at (0, 1.154) {} ;
\node[dot=L30] at (-1.5, -0.867) {} ;
\node[dot=L41] at (1.5, -0.867) {} ;
\node[dot=L32] at (0, 1.731) {} ;
\draw[] (L12)-- (L20) ;
\draw[] (L20)-- (L21) ;
\draw[] (L21)-- (L22) ;
\draw[] (L22)-- (L20) ;
\draw[] (L12)-- (L22) ;
\draw[] (L20)-- (L30) ;
\draw[] (L30)-- (L31) ;
\draw[] (L21)-- (L31) ;
\draw[] (L20)-- (L31) ;
\draw[] (L31)-- (L32) ;
\draw[] (L32)-- (L30) ;
\draw[] (L22)-- (L32) ;
\draw[] (L21)-- (L32) ;
\draw[] (L20)-- (L32) ;
\draw[] (L31)-- (L41) ;
\draw[] (L32)-- (L41) ;
\end{tikzpicture}

\begin{tikzpicture}[line width=.6pt,scale=0.8,dot/.style={circle,inner sep=1.5pt,fill,name=#1}]
\node[dot=L10] at (-0.5, -0.289) {} ;
\node[dot=L22] at (0, 0.577) {} ;
\node[dot=L20] at (-1, -0.577) {} ;
\node[dot=L21] at (1, -0.577) {} ;
\node[dot=L32] at (0, 1.154) {} ;
\node[dot=L30] at (-1.5, -0.867) {} ;
\node[dot=L31] at (1.5, -0.867) {} ;
\node[dot=L42] at (0, 1.731) {} ;
\draw[] (L10)-- (L20) ;
\draw[] (L20)-- (L21) ;
\draw[] (L10)-- (L21) ;
\draw[] (L21)-- (L22) ;
\draw[] (L22)-- (L20) ;
\draw[] (L10)-- (L22) ;
\draw[] (L20)-- (L30) ;
\draw[] (L30)-- (L31) ;
\draw[] (L21)-- (L31) ;
\draw[] (L20)-- (L31) ;
\draw[] (L31)-- (L32) ;
\draw[] (L32)-- (L30) ;
\draw[] (L22)-- (L32) ;
\draw[] (L21)-- (L32) ;
\draw[] (L20)-- (L32) ;
\draw[] (L32)-- (L42) ;
\draw[] (L30)-- (L42) ;
\end{tikzpicture}
\quad
\begin{tikzpicture}[line width=.6pt,scale=0.8,dot/.style={circle,inner sep=1.5pt,fill,name=#1}]
\node[dot=L10] at (-0.5, -0.289) {} ;
\node[dot=L11] at (0.5, -0.289) {} ;
\node[dot=L20] at (-1, -0.577) {} ;
\node[dot=L21] at (1, -0.577) {} ;
\node[dot=L22] at (0, 1.154) {} ;
\node[dot=L30] at (-1.5, -0.867) {} ;
\node[dot=L31] at (1.5, -0.867) {} ;
\node[dot=L32] at (0, 1.731) {} ;
\draw[] (L10)-- (L11) ;
\draw[] (L10)-- (L20) ;
\draw[] (L20)-- (L21) ;
\draw[] (L11)-- (L21) ;
\draw[] (L10)-- (L21) ;
\draw[] (L21)-- (L22) ;
\draw[] (L22)-- (L20) ;
\draw[] (L11)-- (L22) ;
\draw[] (L10)-- (L22) ;
\draw[] (L20)-- (L30) ;
\draw[] (L21)-- (L30) ;
\draw[] (L30)-- (L31) ;
\draw[] (L21)-- (L31) ;
\draw[] (L31)-- (L32) ;
\draw[] (L32)-- (L30) ;
\draw[] (L22)-- (L32) ;
\draw[] (L21)-- (L32) ;
\draw[] (L20)-- (L32) ;
\end{tikzpicture}
\quad
\begin{tikzpicture}[line width=.6pt,scale=0.8,dot/.style={circle,inner sep=1.5pt,fill,name=#1}]
\node[dot=L11] at (0.5, -0.289) {} ;
\node[dot=L12] at (0, 0.577) {} ;
\node[dot=L20] at (-1, -0.577) {} ;
\node[dot=L21] at (1, -0.577) {} ;
\node[dot=L22] at (0, 1.154) {} ;
\node[dot=L40] at (-1.5, -0.867) {} ;
\node[dot=L31] at (1.5, -0.867) {} ;
\node[dot=L32] at (0, 1.731) {} ;
\draw[] (L11)-- (L12) ;
\draw[] (L12)-- (L20) ;
\draw[] (L20)-- (L21) ;
\draw[] (L11)-- (L21) ;
\draw[] (L21)-- (L22) ;
\draw[] (L22)-- (L20) ;
\draw[] (L12)-- (L22) ;
\draw[] (L11)-- (L22) ;
\draw[] (L21)-- (L31) ;
\draw[] (L20)-- (L31) ;
\draw[] (L31)-- (L32) ;
\draw[] (L22)-- (L32) ;
\draw[] (L21)-- (L32) ;
\draw[] (L20)-- (L32) ;
\draw[] (L31)-- (L40) ;
\draw[] (L32)-- (L40) ;

\end{tikzpicture}
\caption{All minimal non-word-representable induced subgraphs in triangulations of GCCG's with three sectors}\label{non-repr-induced-subgraphs}
\end{center}
\end{figure}

\subsubsection{Subdivisions of Triangular Grid Graphs}

The {\em triangular tiling graph} $T^{\infty}$  is the {\em Archimedean tiling} $3^6$ (see Figure~\ref{Tinfty}). By a {\it triangular grid graph} $G$ we mean a graph obtained from $T^{\infty}$ as follows. Specify a finite number of triangles, called {\em cells}, in $T^{\infty}$. The edges of $G$ are then all the edges surrounding the specified cells, while the vertices of $G$ are the endpoints of the edges (defined by intersecting lines in $T^{\infty}$). We say that the specified cells, along with any other cell whose all edges are from $G$, {\em belong} to $G$.

\begin{figure}
\begin{center}
\begin{tikzpicture}[scale=0.5]
\clip (0.3,0.3) rectangle (7.7,5.5);
\foreach \x in {-10,...,10} 
 \foreach \y in {-10,...,10}
 { \fill[black!100] (\x + .5* \y, 0.866*\y) circle(0.5ex);
 }
\foreach \x in {-10,...,10} 
{ \draw (\x,0)-- +(10,17.32);
  \draw (\x,0)-- +(-10,17.32);
}
\foreach \y in {-10,...,10} 
{ \draw (-1,0.866*\y)-- +(10,0);

}
\end{tikzpicture}
\caption{\label{Tinfty} A fragment of the graph  $T^{\infty}$}
\end{center}
\end{figure}
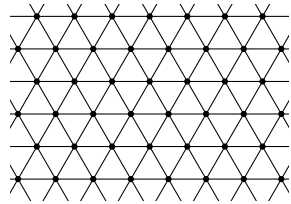

The operation of {\it face subdivision of a cell} is putting a new vertex inside the cell and making it to be adjacent to every vertex of the cell. Equivalently, face subdivision of a cell is replacing the cell (which is the complete graph $K_3$) by a plane version of the complete graph $K_4$.
A {\it face subdivision of a set $S$ of cells} of a triangular grid graph $G$ is a graph obtained from $G$ by subdividing each cell in $S$. The set $S$ of subdivided cells is called a {\it subdivided set}. For example, Figure~\ref{subdiv} shows $K_4$, the face subdivision of a cell, and $A'$, a face subdivision of $A$.

If a face subdivision of $G$ results in a word-representable graph, then the face subdivision is called a {\em word-representable face subdivision}. Also, we say that a word-representable face subdivision of a triangular grid graph $G$ is {\em maximal}
if subdividing any other cell results in a non-word-representable graph.

\newsavebox{\Triangle}
\savebox{\Triangle}
{\begin{tikzpicture}[scale=0.5]
  \draw(0,0)--++(0.5,0.866)-- ++(0.5,-0.866)-- ++(-1,0);
  \fill[black!100] (0,0) circle(0.3ex)
                 ++(0.5,0.866)circle(0.3ex)
                 ++(0.5,-0.866)circle(0.3ex);
 \end{tikzpicture}
}
\newsavebox{\Aa}
\savebox{\Aa}
{\begin{tikzpicture}[scale=0.5]
  \draw (0,0)-- ++(0.5,0.866)-- ++(0.5,0.866)-- ++(0.5,-0.866)-- ++(0.5,-0.866)-- ++(-1,0) 
             -- ++(0.5,0.866)-- ++(-1,0)-- ++(0.5,-0.866)-- ++(-1,0);
  \fill[black!100] (0,0) circle(0.3ex)
                ++(0.5,0.866) circle(0.4ex)
                ++(0.5,0.866) circle(0.4ex)
                ++(0.5,-0.866) circle(0.4ex)
                ++(0.5,-0.866) circle(0.4ex)
                ++(-1,0) circle(0.3ex) ;
 \end{tikzpicture}
}

\begin{figure}
 \begin{center}
\begin{tikzpicture}[scale=0.5]
\clip (0.5,0.5) rectangle (10.5,4.5);
\foreach \x in {-10,...,15} 
 \foreach \y in {-10,...,15}
 { \fill[gray!100] (\x + .5* \y, 0.866*\y) circle(0.3ex);
 }
\foreach \x in {-10,...,15} 
{ \draw[gray,very thin] (\x,0)-- +(10,17.32);
  \draw[gray,very thin] (\x,0)-- +(-10,17.32);
}
\foreach \y in {-10,...,15} 
{ \draw[gray,very thin] (-1,0.866*\y)-- +(15,0);
}

\draw (2,1.732)  node[anchor=south west,inner sep=-0.3pt,outer sep=-0.3pt] {\usebox{\Triangle}}
      --++(0.5,0.2887)
      --++(0.5,-0.2887)
      ++(-0.5,0.2887)--+(0,0.5774);
\fill[black!100] (2,1.732)+(0.5,0.2887) circle(0.3ex);
\draw (5,1.732)  node[anchor=south west,inner sep=-0.3pt,outer sep=-0.3pt] {\usebox{\Aa}}
      +(3,0)  node[anchor=south west,inner sep=-0.3pt,outer sep=-0.3pt]  {\usebox{\Aa}};
\draw (8,1.732)--++(0.5,0.2887)--++(0.5,-0.2887)++(-0.5,0.2887)--++(0,0.5774)
      --++(0.5,-0.2887)--++(0.5,0.2887)++(-0.5,-0.2887)--++(0,-0.5774);
\fill[black!100] (8,1.732)+(0.5,0.2887) circle(0.3ex)
                 +(1,0.5774) circle(0.3ex);
\path (2.5,1) node {$K_4$}
      ++(3.5,0) node {$A$}
      ++(3,0) node {$A'$};
\end{tikzpicture}
\caption{\label{subdiv} Examples of face subdivisions: $K_4$ is the face subdivision of a cell, and $A'$ is a face subdivision of $A$}
\end{center}
\end{figure}
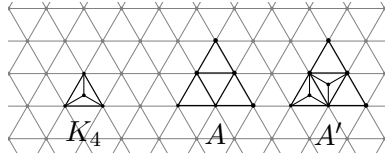

An edge of a triangular grid graph $G$ shared with a cell in $T^{\infty}$ that does not belong to $G$ is called a {\em boundary edge}.  
A cell in $G$ that is incident to at least one boundary edge is called a {\em boundary cell}. A non-boundary cell in $G$ is called an {\em interior cell}.
For example, the boundary edges in the graphs $H$ and $K$ in Figure~\ref{BoundaryEdge} are in bold.

 \newsavebox{\Abig}
 \sbox{\Abig}
 {\begin{tikzpicture}[scale=1]
  \draw[thin](0,0)--++(0.5,0.866)-- ++(0.5,-0.866)-- ++(-1,0);
  \fill[black!100] (0,0) circle(0.3ex)
                 ++(0.5,0.866)circle(0.3ex)
                 ++(0.5,-0.866)circle(0.3ex);
 \end{tikzpicture}
 }
 \newsavebox{\Vbig}
 \sbox{\Vbig}
 {\begin{tikzpicture}[scale=1]
  \draw[thin](0,0)--++(0.5,0.866)-- ++(-1,0)-- ++(0.5,-0.866);
  \fill[black!100] (0,0) circle(0.3ex)
                 ++(0.5,0.866)circle(0.3ex)
                 ++(-1,0)circle(0.3ex);
 \end{tikzpicture}
 }

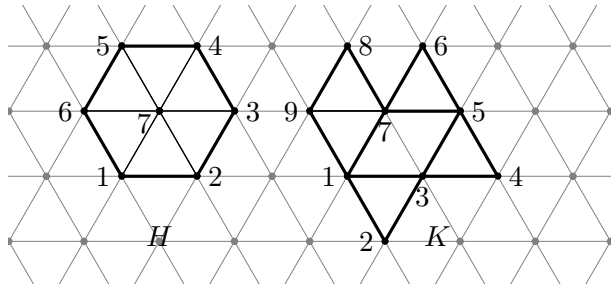
\begin{figure}
 \begin{center}
\begin{tikzpicture}
\clip (-0.5,0.3) rectangle (7.5,4);
\foreach \x in {-10,...,10} 
 \foreach \y in {-10,...,10}
 { \fill[gray!100] (\x + .5* \y, 0.866*\y) circle(0.3ex);
 }
\foreach \x in {-10,...,10} 
{ \draw[gray,very thin] (\x,0)-- +(10,17.32);
  \draw[gray,very thin] (\x,0)-- +(-10,17.32);
}
\foreach \y in {-10,...,10} 
{ \draw[gray,very thin] (-1,0.866*\y)-- +(10,0);
}
\draw (1,1.732)  node[anchor=south west,inner sep=-0.6pt,outer sep=-0.6pt] {\usebox{\Abig}}
      ++(0.5,0)  node[anchor=south west,inner sep=-0.6pt,outer sep=-0.6pt] {\usebox{\Vbig}}
      ++(0,0.866)  node[anchor=south west,inner sep=-0.6pt,outer sep=-0.6pt] {\usebox{\Abig}}
      ++(-0.5,0)  node[anchor=south west,inner sep=-0.6pt,outer sep=-0.6pt] {\usebox{\Vbig}}
      ++(-0.5,0)  node[anchor=south west,inner sep=-0.6pt,outer sep=-0.6pt] {\usebox{\Abig}}
      ++(0,-0.866)  node[anchor=south west,inner sep=-0.6pt,outer sep=-0.6pt] {\usebox{\Vbig}};
\draw[very thick] (1,1.732) node[left]{$1$}--++(1,0)node[right]{$2$}--++(0.5,0.866)node[right]{$3$}--++(-0.5,0.866) node[right]{$4$}
               --++(-1,0)node[left]{$5$}--++(-0.5,-0.866)node[left]{$6$}--++(0.5,-0.866)+(0.3,0.7) node{$7$};

\draw (3+0.5,1.732)  node[anchor=south west,inner sep=-0.6pt,outer sep=-0.6pt] {\usebox{\Vbig}}
      ++(0.5,-0.866)  node[anchor=south west,inner sep=-0.6pt,outer sep=-0.6pt] {\usebox{\Vbig}}
      ++(1,0.866)  node[anchor=south west,inner sep=-0.6pt,outer sep=-0.6pt] {\usebox{\Abig}}
      ++(-0.5,0.866)  node[anchor=south west,inner sep=-0.6pt,outer sep=-0.6pt] {\usebox{\Abig}}
      ++(-1,0)  node[anchor=south west,inner sep=-0.6pt,outer sep=-0.6pt] {\usebox{\Abig}}
      ++(0,-0.866)  node[anchor=south west,inner sep=-0.6pt,outer sep=-0.6pt] {\usebox{\Vbig}};

\draw[very thick] (4,1.732) node[left]{$1$}--++(0.5,-0.866)node[left]{$2$}--++(0.5,0.866)node[below]{$3$}--++(1,0)node[right]{$4$}
                 --++(-0.5,0.866)node[right]{$5$}--++(-0.5,0.866)node[right]{$6$}
                 --++(-0.5,-0.866) node[below]{$7$}--++(-0.5,0.866)node[right]{$8$}--++(-0.5,-0.866) node[left]{$9$}
                 --++(0.5,-0.866);
\draw[very thick] (4,1.732)--++(1,0)--++(0.5,0.866)--++(-1,0)--++(-0.5,-0.866);
\draw(1.5,1.2) node[below] {$H$} +(3.7,0) node[below] {$K$} ;

\end{tikzpicture}

\caption{\label{BoundaryEdge} Graphs $H$ and $K$, where boundary edges are in bold}
\end{center}
\end{figure}

A face subdivision of a triangular grid graph that involves face subdivision of just boundary cells is called a {\it boundary face subdivision}.  The following theorem was proved using the notion of a {\em smart} orientation (see \cite{CKS16} for details).

\begin{theorem}[\cite{CKS16}]\label{subdivtrigrid}
A face subdivision of a triangular grid graph $G$ is word-representable if and only if it has no induced subgraph
 isomorphic to $A''$ in Figure~\ref{nonRepTri}, that is, $G$ has no subdivided interior cell.
\end{theorem}

\begin{figure}
\begin{center}
\begin{tikzpicture}[scale=1.1]
\draw[thick] (0,0) node[anchor=east]{}-- ++(0.5774,1) node[anchor=east] {}--++(0.5774,1) node[anchor=west] {}
                   --++(0.5774,-1) node[anchor=west] {}--++(0.5774,-1) node[anchor=west] {}
                   --++(-2*0.5774,0)--++(-0.5774,1)--++(2*0.5774,0)
                   --++(-0.5774,-1)--+(-2*0.5774,0)
                   (2*0.5774,0)--++(0,2/3) --+(-0.5774,1/3)
                   (2*0.5774,2/3)--+(0.5774,1/3) ;
\path (1.2,0.14) node [anchor=west]{}
     +(0,0.45) node [anchor= west]{};

\fill[black!100] (0,0) circle(0.3ex)
                 ++(0.5774,1) circle(0.3ex)
                 ++(0.5774,1) circle(0.3ex)
                 ++(0.5774,-1) circle(0.3ex)
                 ++(0.5774,-1) circle(0.3ex)
                   (2*0.5774,0) circle(0.3ex)
                 ++(0,2/3) circle(0.3ex);
\end{tikzpicture}
\caption{\label{nonRepTri} The graph $A''$}
\end{center}
\end{figure}
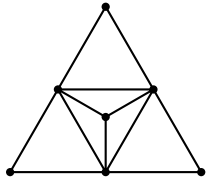

Theorem~\ref{subdivtrigrid} can be applied to the {\em two-dimensional Sierpi\'{n}ski gasket graph $SG(n)$} to find its maximum word-representable subdivision (see \cite{CKS16} for details).

\section{Directions for Further Research}\label{sec6} 

In this section we list some of open problems and directions for further research related to word-representable graphs. The first question though the Reader should ask himself/herself is ``Which graphs in their favourite class of graphs are word-representable?''.

\begin{itemize}
\item Characterize (non-)word-representable planar graphs.
\item Characterize word-representable near-triangulations (containing $K_4$). 
\item Describe graphs representable by words avoiding a pattern $\tau$, where the notion of a ``pattern'' can be specified in any suitable way, e.g. it could be a {\em classical pattern}, a {\em vincular pattern}, or a {\em bivincular pattern} (see \cite{K11} for definitions).
\item Is it true that out of all bipartite graphs on the same number of vertices, {\em crown graphs} require the {\em longest} word-representants?
\item Are there any graphs on $n$ vertices whose representation requires {\em more} than $\lfloor n/2 \rfloor$ copies of each letter?
\item Is the {\em line graph} of a non-word-representable graph {\em always} non-word-representable?
\item Characterize word-representable graphs in terms of {\em forbidden subgraphs}.  
\item Translate a known to you problem on {\em graphs} to {\em words} representing these graphs (assuming such words exist), and find an {\em efficient algorithm} to solve the {\em obtained problem}, and thus the {\em original problem}.  
\end{itemize}

The last two problems are of  fundamental importance.

\end{document}